\documentclass[12pt,twoside,a4paper]{amsart}

\usepackage{amssymb,amsmath,amsthm}
\usepackage{geometry}
\usepackage{color}

\newtheorem{theorem}{Theorem}[section]
\newtheorem{thm}[theorem]{Theorem}

\newtheorem{coro}[theorem]{Corollary}

\theoremstyle{plain}
\newtheorem*{namedthm}{\namedthmname}
\newcounter{namedthm}



\newcommand{\R}{\mathbb{R}}
\newcommand{\C}{\mathbb{C}}
\newcommand{\N}{\mathbb{N}}

\numberwithin{equation}{section}

\usepackage{hyperref}
\hypersetup{
	bookmarks=true,         
	unicode=false,         
	pdftoolbar=true,       
	pdfmenubar=true,       
	pdffitwindow=false,    
	pdfstartview={FitH},   
	pdftitle={Pluripotential vs Viscosity solutions to complex Monge-Amp\`ere flows},    
	pdfauthor={Vincent Guedj, Chinh H. Lu, and Ahmed Zeriahi},     
	colorlinks=true,       
	linkcolor=black,         
	citecolor=black,        
	filecolor=black,      
	urlcolor=black}          

\begin{document}
	\def\K{\mathbb{K}}
	\def\R{\mathbb{R}}
	\def\C{\mathbb{C}}
	\def\Z{\mathbb{Z}}
	\def\Q{\mathbb{Q}}
	\def\D{\mathbb{D}}
	\def\N{\mathbb{N}}
	\def\T{\mathbb{T}}
	\def\P{\mathbb{P}}
	\def\A{\mathscr{A}}
	\def\CC{\mathscr{C}}
	%%%%%%%%%%%%%%%%%%%%%%%%%%%%%%%%%%%%%%%%%%%%%%%%%
	\renewcommand{\theequation}{\thesection.\arabic{equation}}
	\renewenvironment{proof}{{\bfseries Proof:}}{\qed}
	\renewcommand{\thelemme}{\empty{}}
	\newtheorem{preuv}{Preuve}[section]
	\newtheorem{cond}{C}
	\newtheorem{lemma}{Lemma}[section]
	\newtheorem{corollary}{Corollary}[section]
	\newtheorem{proposition}{Proposition}[section]
	\newtheorem{notation}{Notation}[section]
	\newtheorem{remark}{Remarque}[section]
	\newtheorem{example}{Example}[section]
	\newtheorem{probleme}{Probleme}[section]
	\bibliographystyle{plain}

\title[Generalized study of the operator $\alpha \partial^k \bar{\partial}^{k} + \beta \bar{\partial}^k +\gamma  \partial^k +  c$~~]{\textbf {Generalized study of the operator} $\alpha \partial^k \bar{\partial}^{k} + \beta \bar{\partial}^k +\gamma  \partial^k +  c$ \textbf {in weighted Hilbert space} $L^2(\mathbb{\C}, \mathrm{\textnormal{e}}^{\mathrm{-\vert z \vert^2}})$}
\author[E.\  Bodian   \& W.\ O.\  Ingoba \& S.\  Sambou \&   P.\  Badiane  \& S.\  Sambou ]{ Eramane Bodian  \& Winnie Ossete Ingoba \& Souhaibou Sambou \& Papa Badiane \&   Salomon Sambou }
\address{Department of Mathematics\\UFR of Sciences and Technologies \\ University Assane Seck of Ziguinchor, BP: 523 (S\'en\'egal)}
\email{m.bodian@univ-zig.sn}
\address{Departement of Mathematics\\ FST \\ Marien N'Gouabi University of Brazzaville, BP: 69 (Congo)}
\email{wnnossete@gmail.com}
\address{Department of Mathematics\\UFR of Sciences and Technologies \\ University Assane Seck of Ziguinchor, BP: 523 (S\'en\'egal)}
\email{s.sambou1440@zig.univ.sn }
\address{Department of Mathematics\\UFR of Sciences and Technologies \\ University Assane Seck of Ziguinchor, BP: 523 (S\'en\'egal)}
\email{p.badiane4963@zig.univ.sn}
\address{D\'epartement de Math\'ematiques\\UFR des Sciences et Technologies \\ Universit\'e Assane Seck de Ziguinchor, BP: 523 (S\'en\'egal)}
\email{wnnossete@gmail.com}

\address{Department of Mathematics\\UFR of Sciences and Technologies \\ University Assane Seck of Ziguinchor, BP: 523 (S\'en\'egal)}
\email{ssambou@univ-zig.sn }
\date{\today}


\maketitle
%%%%%%%%%%%%%%%%%%%%%%%%%%%%%%%%%%%%%%%%%%%%%%%%%%%%%%%%%%%%%%%%%%%%%%%%%%%%%%%%%%%%%%%%%%%%%%%%%%%%%%%
\renewcommand{\abstractname}{Abstract}
\begin{abstract}
By H\"ormander's $L^2$-method, we study the operator $\alpha \partial^k \bar{\partial}^{k} + \beta \bar{\partial}^k +\gamma  \partial^k +  c$ for any order $k$ with $\alpha, \beta, \gamma \in \mathbb{R}$ such that $(\alpha, \beta, \gamma) \neq(0,0,0)$ in the weighted Hilbert space $L^2(\mathbb{\C}, \mathrm{\textnormal{e}}^{-|z|^2})$. We prove the existence of its right inverse which is also a bounded operator. Subsequently we will study two cases that arise from this operator, namely:
\begin{enumerate}
\item Case where $\alpha= \gamma=0$:  The operator $\beta \bar{\partial}^{k}  + c$ with $\vert \beta \vert \geq 1$.
\item Case where $\beta= \gamma=0$: The operator  $\alpha \partial^{k} \bar{\partial}^{k}  + c$ with $\vert \alpha \vert \geq 1$.

\end{enumerate}
%%%%%%%%%%%%%%%%%%%%%%%%%%%%%%%%%%%%%%%%%%%%%%%%%%%%%%%%%%%%%%%%%%%%%%%%%%%%%%%%%%%%%%%%%%%%%%%%%%%%%%%%%%
\vskip 2mm
\noindent
\vskip 1.3mm
\noindent
%\textit{
\end{abstract}  
\renewcommand{\abstractname}{Abstract}
%%%%%%%%%%%%%%%%%%%%%%%%%%%%%%%%%%%%%%%%%%%%%%%%%%%%%%%%%%%%%%%%%%%%%%%%%%%%%%%%%%%%%%%%%%%%%%%%%%%%%%%%%%   
\section*{Statements and Declarations}
Competing Interests: The authors attest that there are  non-financial interests that are directly or indirectly related to the work submitted for publication.
\section*{Data availability statements}
We have not included any data in this manuscript.
\section{Introduction}
In {\cite{1}}, Shoayu DAI and Yifei PAN are interested by the study of the
operator $\frac{d^k}{dx^k} + a$. They have extended this work in {\cite{2}} to
the operator $\bar{\partial}^k + a$. In same direction, inspired by the
$\partial \bar{\partial}$-problem, We have studied the operator $\partial^k
\bar{\partial}^k + c$ in the weighted Hilbert space $L^2 (\mathbb{C},
\mathrm{\text{e}}^{- |z|^2})$ in {\cite{4}}. In this paper, we consider the
operator $\alpha \partial^k \bar{\partial}^k + \beta \bar{\partial}^k + \gamma
\partial^k + c$ so that if $\alpha = \gamma = 0$ and $\beta = 1,$ we get
the result of Shoayu DAI and Yifei PAN in {\cite{2}}. On the other hand, if
$\alpha = \gamma = 0$ and $\alpha = 1$, then we get our result in
{\cite{4}}. So if the previous situations do not apply, we will have a general
study of this operator $\alpha \partial^k \bar{\partial}^k + \beta
\bar{\partial}^k + \gamma \partial^k + c$ in weighted Hilbert space $L^2
(\mathbb{C}, \mathrm{\text{e}}^{- |z|^2})$ while noting that we do not
consider the powers of the Laplacian in the complex plane in the case of the
$L^2$-estimate. Under certain hypothesis, we prove the existence of a weak
solution of the equation
\[ \alpha \partial^k \bar{\partial}^k u + \beta \bar{\partial}^k u + \gamma
   \partial^k u + cu = f \]
where $u, f \in L^2 (\mathbb{C}, \mathrm{\text{e}}^{- |z|^2})$ and $c$ is a
complex constant. We thus have the following result:

\begin{theorem}
  \label{P}Let $f \in L^2 (\mathbb{C}, \mathrm{\text{e}}^{- |z|^2})$, $\alpha,
  \beta, \gamma \in \mathbb{R}$ with $(\alpha, \beta, \gamma) \neq (0, 0, 0)$
  and $\phi : \mathbb{C} \rightarrow \mathbb{C}$ a function of class
  $C^{\infty}$ with compact support such as
  \[ \gamma \langle \bar{\partial}^k (\partial^{k - l} \phi), \partial^{k - l}
     \phi \rangle_{\varphi} + \beta \langle \partial^k (\bar{\partial}^{k - l}
     \phi), \bar{\partial}^{k - l} \phi \rangle_{\varphi} = 0, \]
  then there exists a weak solution $u \in L^2 (\mathbb{C},
  \mathrm{\text{e}}^{- |z|^2})$ of the equation
  \[ \alpha \partial^k \bar{\partial}^k u + \beta \bar{\partial}^k u + \gamma
     \partial^k u + cu = f \]
  with the norm estimate
  \[ \int_{\mathbb{C}} |u|^2  \mathrm{\text{e}}^{- |z|^2} d \sigma \leq
     \frac{1}{\alpha^2 (k!)^2 + \beta^2 (k!) + \gamma^2 (k!)} 
     \int_{\mathbb{C}} |f|^2  \mathrm{\text{e}}^{- |z|^2} d \sigma . \]
\end{theorem}

As a consequence of Theorem \ref{P}, we show that the operator $\alpha
\partial^k \bar{\partial}^k + \beta \bar{\partial}^k + \gamma \partial^k + c$
has a bounded right inverse in $L^2 (\mathbb{C}, \mathrm{\text{e}}^{-
|z|^2})$. We conclude the study in the following cases :
\begin{enumerate}
  \item Case where $\alpha = \gamma = 0$ : The operator $\beta
  \bar{\partial}^k + c$ with $| \beta | \geq 1$.
  
  \item Case where $\beta = \gamma = 0$ : The operator $\alpha \partial^k
  \bar{\partial}^k + c$ with $| \alpha | \geq 1$.
\end{enumerate}
\section{Some Lemmas}
Consider the weighted Hilbert space

$$ L^2(\mathbb{\C}, \mathrm{\textnormal{e}}^{- \varphi}) = \left\lbrace  f \in L_{loc}^2(\mathbb{\C}) : \int_{\mathbb{C}} \vert f \vert^2 \mathrm{\textnormal{e}}^{- \varphi } d\sigma < + \infty \right\rbrace $$
where $\varphi$ is a positive function on $\mathbb{C}$.\\
We define weighted inner product   on $L^2(\mathbb{\C}, \mathrm{\textnormal{e}}^{- \varphi})$ by : 
$$<f, g>_\varphi = \int_\mathbb{C} \overline{f}g \mathrm{\textnormal{e}}^{- \varphi} d\sigma \qquad \forall f, g \in L^2(\mathbb{\C}, \mathrm{\textnormal{e}}^{- \varphi})$$
and the norm on $L^2(\mathbb{\C}, \mathrm{\textnormal{e}}^{- \varphi})$ by :
$$\parallel f \parallel_\varphi = \sqrt{<f, f>_\varphi} \; \; \; \; \; \; \; \; \forall \; \; \; \; f \in L^2(\mathbb{\C}, \mathrm{\textnormal{e}}^{- \varphi}).$$
We denote by $C_0^\infty(\mathbb{C})$ the space of functions $\phi : \mathbb{C} \rightarrow \mathbb{C}$ of class $C^\infty$ with compact support.\\
Let $u,  f \in L_{loc}^2(\mathbb{\C})$. 
\vskip 0.5cm
\begin{enumerate}
\item We say that $u$ is a weak solution of the equation $\partial^k \bar{\partial}^{k} u = f$ if
$$ \int_\mathbb{C} u \partial^k \bar{\partial}^{k} \phi d\sigma = \int_\mathbb{C} f \phi d\sigma .$$
\item We say that $u$ is a weak solution of the equation $ \bar{\partial}^{k} u = f$ if
$$ \int_\mathbb{C} u  \bar{\partial}^{k} \phi d\sigma = (-1)^k \int_\mathbb{C} f \phi d\sigma .$$
\item We say that $u$ is a weak solution of the equation $ \partial^{k} u = f$ if
$$ \int_\mathbb{C} u  \partial^{k} \phi d\sigma = (-1)^k \int_\mathbb{C} f \phi d\sigma .$$
\end{enumerate}
\vskip 0.5cm
Let $\varphi$ be a $C^\infty$ positive function on $\mathbb{C}$ and $\phi \in C_0^\infty(\mathbb{C})$, we define the formal adjoint of $\partial^k \bar{\partial}^{k}$ with respect to weighted inner product   defined in $L^2(\mathbb{\C}, \mathrm{\textnormal{e}}^{- \varphi})$ and the definition of the weak solution as follows :
\begin{eqnarray*}
<\phi, \partial^k \bar{\partial}^{k} u> & = & \int_{\mathbb{C}} \overline{\phi} \partial^k \bar{\partial}^{k} u \mathrm{\textnormal{e}}^{- \varphi} d\sigma\\
& = &  \int_{\mathbb{C}}  \bar{\partial}^{k} \partial^k( \overline{\phi}\mathrm{\textnormal{e}}^{- \varphi})  u  d\sigma\\
& = &  \int_{\mathbb{C}} \mathrm{\textnormal{e}}^{ \varphi} \bar{\partial}^{k} \partial^k ( \overline{\phi}\mathrm{\textnormal{e}}^{- \varphi})  u \mathrm{\textnormal{e}}^{- \varphi} d\sigma\\
& = & \int_{\mathbb{C}} \overline{\mathrm{\textnormal{e}}^{ \varphi} \partial^{k} \bar{\partial}^k ( \phi \mathrm{\textnormal{e}}^{- \varphi})} u \mathrm{\textnormal{e}}^{- \varphi} d\sigma,\;\; \mathrm{\textnormal{e}}^{- \varphi}\; \mbox{being with real variables}\\
& = & <\mathrm{\textnormal{e}}^{ \varphi} \partial^{k} \bar{\partial}^k ( \phi \mathrm{\textnormal{e}}^{- \varphi}), u >_\varphi\\
& =: & < \bar{\partial}_\varphi^{k*} \partial_\varphi^{k*}( \phi), u >_\varphi
\end{eqnarray*}
where $ \bar{\partial}_\varphi^{k*} \partial_\varphi^{k*}( \phi) = \mathrm{\textnormal{e}}^{ \varphi} \partial^{k} \bar{\partial}^k ( \phi \mathrm{\textnormal{e}}^{- \varphi})$ is the formal adjoint of $\partial^k \bar{\partial}^{k}$ in $C_0^\infty(\mathbb{C})$.\\
By analogy we have
$$
\left\{
    \begin{array}{lll}
        \bar{\partial}_\varphi^{k*}\phi = (-1)^k\mathrm{\textnormal{e}}^{ \varphi}\partial^k( \phi \mathrm{\textnormal{e}}^{- \varphi})\\
        \mbox{and}\\
        \partial_\varphi^{k*}\phi = (-1)^k\mathrm{\textnormal{e}}^{ \varphi}\bar{\partial}^k( \phi \mathrm{\textnormal{e}}^{- \varphi}) & 
    \end{array}
\right.
$$
the formal adjoints of the operators $\bar{\partial}^k$ and $\partial^k$.\\
Let $(\alpha \partial^k \bar{\partial}^{k} + \beta \bar{\partial}^k +\gamma  \partial^k +  c)_\varphi^*$ be the formal adjoints of $\alpha \partial^k \bar{\partial}^{k} + \beta \bar{\partial}^k +\gamma  \partial^k +  c$ in $C_0^\infty(\mathbb{C})$.\\
Let $I_\varphi^* = I$ where $I$ is the identity operator, then $$ (\alpha \partial^k \bar{\partial}^{k} + \beta \bar{\partial}^k +\gamma  \partial^k +  c)_\varphi^* = \alpha \bar{\partial}_\varphi^{k*} \partial_\varphi^{k*} + \beta \bar{\partial}_\varphi^{k*} + \gamma   \partial_\varphi^{k*}  + c.$$
We start by giving some lemmas in the general case of weighted spaces based on
functional analysis which will be very useful to us in the proof of Theorem \ref{P}.
\vskip 0.5cm
For the sequel, let us set :
$$
\partial^{k} \bar{\partial}^k = R\;\;\mbox{et}\;\; \bar{\partial}_\varphi^{k*} \partial_\varphi^{k*} = \left(\partial^{k} \bar{\partial}^k\right)^* = R^*. 
$$
\begin{lemma} \label{A}
For any function $f \in L^2(\mathbb{\C}, \mathrm{\textnormal{e}}^{- \varphi})$, there exists a weak solution $u \in L^2(\mathbb{\C}, \mathrm{\textnormal{e}}^{- \varphi})$ of the equation
$$ \alpha Ru + \beta \bar{\partial}^ku +\gamma  \partial^ku +  cu  = f$$ with the estimation of the norm
$$\vert \vert u \vert \vert_\varphi^2 \leq a$$ if and only if
$$\vert <f, \phi>_\varphi \vert^2 \leq a \vert \vert (\alpha R + \beta \bar{\partial}^k +\gamma  \partial^k +  c)_\varphi^* \phi \vert \vert_\varphi ^2$$
$\forall$ $\phi \in C_0^\infty(\mathbb{C})$ where $a$ is a constant.
\end{lemma}
\begin{proof}
Let's set $H =  \alpha R + \beta \bar{\partial}^k +\gamma  \partial^k +  c$ and $H_\varphi^* = (\alpha R + \beta \bar{\partial}^k +\gamma  \partial^k +  c)_\varphi^*$.\\
Suppose that there exists  $u \in L^2(\mathbb{\C}, \mathrm{\textnormal{e}}^{- \varphi})$ such that $$ \alpha Ru + \beta \bar{\partial}^ku +\gamma  \partial^ku +  cu  = f$$ with the estimation of the norm
$$\vert \vert u \vert \vert_\varphi^2 \leq a$$ 
then we have
\begin{eqnarray*}
\vert <f, \phi>_\varphi \vert^2 = \vert <Hu, \phi>_\varphi \vert^2 & = & \vert <u, H_\varphi^* \phi>_\varphi \vert^2\\
& \leq & \vert \vert u \vert \vert_\varphi^2 \vert \vert H_\varphi^* \phi \vert \vert_\varphi^2, \;\; \mbox{by Cauchy-Schwarz}\\
& \leq & a \vert \vert H_\varphi^* \phi \vert \vert_\varphi^2 \\
\end{eqnarray*}
so  $$\vert <f, \phi>_\varphi \vert^2 \leq a \vert \vert (\alpha R + \beta \bar{\partial}^k +\gamma  \partial^k +  c)_\varphi^* \phi \vert \vert_\varphi ^2.$$
suppose that $$\vert <f, \phi>_\varphi \vert^2 \leq a \vert \vert (\alpha R + \beta \bar{\partial}^k +\gamma  \partial^k +  c)_\varphi^* \phi \vert \vert_\varphi ^2.$$
Consider the subspace
$$ E = \{ H_\varphi^* \phi : \phi \in C_0^\infty(\mathbb{C}) \} \subset L^2(\mathbb{\C}, \mathrm{\textnormal{e}}^{- \varphi}).$$
We define a linear map by
$$ L_f : E \rightarrow \mathbb{C}$$ with $$L_f(H_\varphi^* \phi) = <f, \phi>_\varphi = \int_\mathbb{C} \overline{f} \phi \mathrm{\textnormal{e}}^{- \varphi} d \sigma$$
or
$$\vert L_f(H_\varphi^* \phi) \vert = \vert <f, \phi>_\varphi \vert \leq \sqrt{a} \vert \vert H_\varphi^* \phi \vert \vert_\varphi$$
 $L_f$ is bounded in $E$.\\
By the Hahn-Banach Extension Theorem, $L_f$ can be extended into an operator $\tilde{L}_f$ on $L^2(\mathbb{\C}, \mathrm{\textnormal{e}}^{- \varphi})$ with
$$\vert \tilde{L}_f(g)\vert \leq \sqrt{a} \vert \vert g \vert \vert_\varphi \; \; \; \forall \; \; g \in L^2(\mathbb{\C}, \mathrm{\textnormal{e}}^{- \varphi}).$$
According to the Riesz representation Theorem, there exists a unique $u_0 \in L^2(\mathbb{\C}, \mathrm{\textnormal{e}}^{- \varphi})$ such that
$$\tilde{L}_f(g) = <u_0, g>_\varphi \; \; \; \forall \; \; g \in L^2(\mathbb{\C}, \mathrm{\textnormal{e}}^{- \varphi})$$
so
$$\tilde{L}_f(H_\varphi^* \phi) = <u_0, H_\varphi^* \phi>_\varphi = <Hu_0,  \phi>_\varphi$$
or 
$$\tilde{L}_f(H_\varphi^* \phi) = L_f(H_\varphi^* \phi) =  <f,  \phi>_\varphi$$ on $E$ so
$$<Hu_0,  \phi>_\varphi = <f,  \phi>_\varphi$$
thus $$ Hu_0 = f$$
which implies
$$\alpha Ru_0  + \beta \bar{\partial}^ku_0  +\gamma  \partial^ku_0 + cu_0 = f$$
and
$$\vert \vert u_0 \vert \vert_\varphi ^2 = \vert  <u_0 ,  u_0 >_\varphi \vert = \vert \tilde{L}_f(u_0)\vert \leq \sqrt{a} \vert \vert u_0 \vert \vert_\varphi$$
so
$$ \vert \vert u_0 \vert \vert_\varphi ^2 \leq a$$
or $u_0 \in  L^2(\mathbb{\C}, \mathrm{\textnormal{e}}^{- \varphi})$
then let $u_0 = u$, so there exists  $u \in L^2(\mathbb{\C}, \mathrm{\textnormal{e}}^{- \varphi})$ such that
$$\alpha Ru + \beta \bar{\partial}^ku +\gamma  \partial^ku + cu = f$$
with
$$\vert \vert u \vert \vert_\varphi ^2 \leq a.$$
\end{proof}

\begin{lemma} \label{B} \item
$\vert \vert (\alpha R + \beta \bar{\partial}^k +\gamma  \partial^k  + c)_\varphi^* \phi \vert \vert_\varphi^2 = \vert \vert (\alpha R + \beta \bar{\partial}^k +\gamma  \partial^k  + c) \phi \vert \vert^2_\varphi + \alpha^2 <\phi,R( R^* \phi) - R^*(R \phi) >_\varphi +  \beta^2 <\phi, \bar{\partial}^{k}( \bar{\partial}^{k*} \phi) - \bar{\partial}^{k*} ( \bar{\partial}^{k} \phi) >_\varphi +  \gamma^2 <\phi, \partial^{k}( \partial^{k*} \phi) - \partial^{k*} ( \partial^{k} \phi) >_\varphi +  \alpha \beta <\phi,R( \bar{\partial}^{k*} \phi) +\bar{\partial}^{k}( R^* \phi)  - R^*( \bar{\partial}^{k} \phi) - \bar{\partial}^{k*}( R \phi) >_\varphi +  \alpha \gamma <\phi,R( \partial^{k*} \phi) +\partial^{k}( R^* \phi)  - R^*( \partial^{k} \phi) - \partial^{k*}( R \phi) >_\varphi  +  \beta \gamma <\phi, \bar{\partial}^{k}( \partial^{k*} \phi) +\partial^{k}( \bar{\partial}^{k*}  \phi)  - \bar{\partial}^{k*} ( \partial^{k} \phi) - \partial^{k*}( \bar{\partial}^{k} \phi) >_\varphi $\\
$\forall$ $\phi \in C_0^\infty(\mathbb{C})$.
\end{lemma}
\begin{proof}
Let's set $H = \alpha \partial^k \bar{\partial}^{k} + \beta \bar{\partial}^k +\gamma  \partial^k  + c$ and \\$H_\varphi^* = (\alpha \partial^k \bar{\partial}^{k} + \beta \bar{\partial}^k +\gamma  \partial^k  + c)_\varphi^*$.\\
For all $\phi \in C_0^\infty(\mathbb{C})$
\begin{eqnarray*}
\vert \vert H_\varphi^* \phi \vert \vert^2_\varphi & = &  <H_\varphi^* \phi, H_\varphi^* \phi>_\varphi \\
& = &  < \phi, H H_\varphi^* \phi>_\varphi \\ 
& = &  < \phi,  H_\varphi^* H  \phi + H H_\varphi^* \phi - H_\varphi^* H  \phi>_\varphi\\
& = &  <H \phi,  H  \phi> + <\phi, H H_\varphi^* \phi - H_\varphi^* H  \phi>_\varphi \\
& = &   \vert \vert H \phi \vert \vert^2_\varphi + <\phi, H H_\varphi^* \phi - H_\varphi^* H_ \phi>_\varphi \\
\end{eqnarray*}
We have\\
\begin{eqnarray*}
H H_\varphi^* \phi & = & (\alpha R + \beta \bar{\partial}^k +\gamma  \partial^k  + c)(\alpha R + \beta \bar{\partial}^k +\gamma  \partial^k  + c)_\varphi^* \phi  \\
& = &  (\alpha R + \beta \bar{\partial}^k +\gamma  \partial^k  + c)( \alpha R^* + \beta \bar{\partial}^{k*} +\gamma  \partial^{k*}  + c) \phi  \\ & = & \alpha^2 R( R^* \phi) + \alpha \beta R( \bar{\partial}_\varphi^{k*}  \phi) + \alpha \gamma R( \partial_\varphi^{k*}  \phi)  + \alpha c R \phi + \alpha \beta \bar{\partial}^{k}( R^*  \phi) +\beta^2 \bar{\partial}^k(\bar{\partial}_\varphi^{k*}\phi) \\ & & + \beta \gamma \bar{\partial}^k(\partial_\varphi^{k*}\phi) + \beta c\bar{\partial}^{k}\phi + \gamma \alpha \partial^{k}(R^*\phi) + \gamma \beta  \partial^{k}(\bar{\partial}_\varphi^{k*} \phi) + \gamma^2 \partial^{k}( \partial_\varphi^{k*} \phi) + \gamma c \partial^{k} + \\ & & \alpha c R^* \phi + \beta c \bar{\partial}_\varphi^{k*} \phi  + \gamma c \partial_\varphi^{k*} \phi+ c^2 \phi  \\
\end{eqnarray*}
Similary we have, 
\begin{eqnarray*}
H_\varphi^* H  \phi & = &  \alpha^2  R^*( R \phi) + \alpha \beta R^*( \bar{\partial}_\varphi^{k}  \phi) + \alpha \gamma R^*( \partial_\varphi^{k}  \phi)  + \alpha c R^* \phi + \alpha \beta \bar{\partial}^{k*}( R  \phi) + \beta^2 \bar{\partial}^{k*}(\bar{\partial}_\varphi^{k}\phi) \\ & &  + \beta \gamma \bar{\partial}^{k*}(\partial_\varphi^{k}\phi)+ \beta c\bar{\partial}^{k*}  \phi  + \gamma \alpha \partial^{k*}(R \phi) + \gamma \beta  \partial^{k*}(\bar{\partial}_\varphi^{k} \phi)  + \gamma^2 \partial^{k*}( \partial_\varphi^{k} \phi)  + \gamma c \partial^{k*} \\ & &  + \alpha c R \phi + \beta c \bar{\partial}_\varphi^{k} \phi  + \gamma c \partial_\varphi^{k} \phi+ c^2 \phi  
\end{eqnarray*}
so
\begin{eqnarray*}
H H_\varphi^* \phi - H_\varphi^* H  \phi & = & \alpha^2 (R( R^* \phi) - R^*(R \phi)) + \beta^2 ( \bar{\partial}^{k}( \bar{\partial}^{k*} \phi) - \bar{\partial}^{k*} ( \bar{\partial}^{k} \phi) ) + \gamma^2 ( \partial^{k}( \partial^{k*} \phi) - \partial^{k*} ( \partial^{k} \phi) ) \\ & &  + \alpha \beta (R( \bar{\partial}^{k*} \phi) +\bar{\partial}^{k}( R^* \phi) - R^*( \bar{\partial}^{k} \phi) - \bar{\partial}^{k*}( R \phi)) + \alpha \gamma (R( \partial^{k*} \phi) +\partial^{k}( R^* \phi) \\ & &  - R^*( \partial^{k} \phi)  - \partial^{k*}( R \phi)) + \beta \gamma( \bar{\partial}^{k}( \partial^{k*} \phi) +\partial^{k}( \bar{\partial}^{k*}  - \bar{\partial}^{k*} ( \partial^{k} \phi) - \partial^{k*}( \bar{\partial}^{k} \phi)) \phi)  
\end{eqnarray*}
thus
\begin{eqnarray*}
\vert \vert H_\varphi^* \phi \vert \vert^2_\varphi  & = & \vert \vert H \phi \vert \vert^2_\varphi +  \alpha^2 <\phi, R( R^* \phi) - R^*(R \phi) >_\varphi + \beta^2 <\phi, \bar{\partial}^{k}( \bar{\partial}^{k*} \phi) - \bar{\partial}^{k*} ( \bar{\partial}^{k} \phi) >_\varphi \\ & &   + \gamma^2 <\phi, \partial^{k}( \partial^{k*} \phi) - \partial^{k*} ( \partial^{k} \phi) >_\varphi + \alpha \beta <\phi, R( \bar{\partial}^{k*} \phi) +\bar{\partial}^{k}( R^* \phi)  - R^*( \bar{\partial}^{k} \phi) \\ & &   - \bar{\partial}^{k*}( R \phi) >_\varphi + \alpha \gamma <\phi,R( \partial^{k*} \phi) +\partial^{k}( R^* \phi)  - R^*( \partial^{k} \phi) - \partial^{k*}( R \phi) >_\varphi + \\ & & \beta  \gamma <\phi, \bar{\partial}^{k}( \partial^{k*} \phi) +\partial^{k}( \bar{\partial}^{k*}  \phi)  - \bar{\partial}^{k*} ( \partial^{k} \phi) - \partial^{k*}( \bar{\partial}^{k} \phi) >_\varphi 
\end{eqnarray*}
thus
%\scriptsize{
\begin{eqnarray*}
\vert \vert (\alpha R + \beta \bar{\partial}^k +\gamma  \partial^k  + c)_\varphi^* \phi \vert \vert^2_\varphi & = & \vert \vert \alpha R + \beta \bar{\partial}^k +\gamma  \partial^k  + c \phi \vert \vert^2_\varphi + \alpha^2 <\phi,R( R^* \phi) - R^*(R \phi) >_\varphi + \\ & & \beta^2 <\phi, \bar{\partial}^{k}( \bar{\partial}^{k*} \phi) - \bar{\partial}^{k*} ( \bar{\partial}^{k} \phi) >_\varphi + \gamma^2 <\phi, \partial^{k}( \partial^{k*} \phi) - \partial^{k*} ( \partial^{k} \phi) >_\varphi \\ & &  + \alpha \beta<\phi, R( \bar{\partial}^{k*} \phi) +\bar{\partial}^{k}( R^* \phi)  - R^*( \bar{\partial}^{k} \phi) - \bar{\partial}^{k*}( R \phi) >_\varphi + \\ & & \alpha \gamma <\phi,R( \partial^{k*} \phi) +\partial^{k}( R^* \phi)  - R^*( \partial^{k} \phi) - \partial^{k*}( R \phi) >_\varphi  +  \\ & & \beta \gamma <\phi, \bar{\partial}^{k}( \partial^{k*} \phi) +\partial^{k}( \bar{\partial}^{k*}  \phi)  - \bar{\partial}^{k*} ( \partial^{k} \phi) - \partial^{k*}( \bar{\partial}^{k} \phi) >_\varphi 
\end{eqnarray*}
%}
\end{proof}
\begin{lemma} \label{C}
For all $\phi \in C_0^\infty(\mathbb{C})$, we have by setting : 

$$
\left\{
    \begin{array}{lll}
        C_k^i  C_k^j C_k^l = O_{i,j,l,k}\\
        \mbox{and}\\
        C_k^i  C_k^j = O_{i,j,k} & 
    \end{array}
\right.
$$

\begin{eqnarray*}
1) <\phi,  R( \partial_\varphi^{k*}  \phi) - \partial_\varphi^{k*} (  R \phi)>_\varphi & = & <\phi, (-1)^k \sum_{i,j,l=1}^k O_{i,j,l,k}  (\partial^{k-l}\bar{\partial}^{k-i}\bar{\partial}^{k-j}\phi)\times \partial^l \bar{\partial}^j (\mathrm{\textnormal{e}}^{ \varphi} \bar{\partial}^i  \mathrm{\textnormal{e}}^{- \varphi})>\\
2) <\phi,  R( \bar\partial_\varphi^{k*}  \phi) - \bar\partial_\varphi^{k*} (  R \phi)>_\varphi & = & <\phi, (-1)^k \sum_{i,j,l=1}^k O_{i,j,l,k} (\partial^{k-l}\partial^{k-i}\bar{\partial}^{k-j}\phi)\times \partial^l \bar{\partial}^j (\mathrm{\textnormal{e}}^{ \varphi} \partial^i  \mathrm{\textnormal{e}}^{- \varphi})>\\
3) <\phi,  \partial^k ( R^* \phi) - R^*(  \partial^k  \phi)>_\varphi & = & <\phi,  \sum_{i,j,l=1}^k O_{i,j,l,k} (\partial^{k-l}\partial^{k-i}\bar{\partial}^{k-j}\phi) \times \partial^l (\mathrm{\textnormal{e}}^{ \varphi} \partial^i \bar{\partial}^j  \mathrm{\textnormal{e}}^{- \varphi})>\\
4) <\phi,  \bar{\partial}^k ( R^* \phi) - R^*(  \bar{\partial}^k  \phi)>_\varphi & = & <\phi,  \sum_{i,j=,l=1}^k O_{i,j,l,k}  (\bar{\partial}^{k-l}\partial^{k-i}\bar{\partial}^{k-j}\phi)  \bar{\partial}^l (\mathrm{\textnormal{e}}^{ \varphi} \partial^i \bar{\partial}^j  \mathrm{\textnormal{e}}^{- \varphi})>\\
5) <\phi,  \partial^k ( \bar{\partial}_\varphi^{k*} \phi) - \bar{\partial}_\varphi^{k*} (  \partial^k  \phi)>_\varphi & = & <\phi, (-1)^k  \sum_{i,j=1}^k  O_{i,j,k}  (\partial^{k-j} \partial^{k-i}\phi) \times \partial^j (\mathrm{\textnormal{e}}^{ \varphi} \partial^i  \mathrm{\textnormal{e}}^{- \varphi})>\\
6) <\phi,  \bar{\partial}^k ( \partial_\varphi^{k*} \phi) - \partial_\varphi^{k*} (  \bar{\partial}^k  \phi)>_\varphi & = & <\phi, (-1)^k  \sum_{i,j=1}^k  O_{i,j,k} (\bar{\partial}^{k-j} \bar{\partial}^{k-i}\phi) \times  \bar{\partial}^j (\mathrm{\textnormal{e}}^{ \varphi} \bar{\partial}^i  \mathrm{\textnormal{e}}^{- \varphi})> .
  \end{eqnarray*}
\end{lemma}
\begin{proof}\vskip 0.5mm
1)Let recall that :
$$
\partial_\varphi^{k*} ( \phi)  = (-1)^k \mathrm{\textnormal{e}}^{ \varphi}  \bar{\partial}^k ( \phi \mathrm{\textnormal{e}}^{- \varphi})\;\;\mbox{with}\;\;\bar{\partial}^k ( \phi \mathrm{\textnormal{e}}^{- \varphi}) = \sum_{i=0}^k C_k^i    \bar{\partial}^{k-i}\phi (\bar{\partial}^i \mathrm{\textnormal{e}}^{- \varphi}) 
$$
$$
\Rightarrow \bar\partial_\varphi^{k*} ( \phi) = (-1)^k  \sum_{i=0}^k  C_k^i  \bar{\partial}^{k-i}\phi (\mathrm{\textnormal{e}}^{ \varphi} \bar{\partial}^i  \mathrm{\textnormal{e}}^{- \varphi})\;\;\mbox{and}\;\;\partial^k\left(\partial_\varphi^{k*} ( \phi)\right) = \bar{\partial}^{k}\left[ (-1)^k  \sum_{i=0}^k  C_k^i  \bar{\partial}^{k-i}\phi (\mathrm{\textnormal{e}}^{ \varphi} \bar{\partial}^i  \mathrm{\textnormal{e}}^{- \varphi})\right] .
$$
By applying the same principle of derivation to the order $k$ of the product of two functions, we have :
\begin{eqnarray*}
\bar\partial^k\left(\partial_\varphi^{k*} ( \phi)\right) & = & \bar{\partial}^{k}\left[ (-1)^k  \sum_{i=0}^k  C_k^i  \bar{\partial}^{k-i}\phi (\mathrm{\textnormal{e}}^{ \varphi} \bar{\partial}^i  \mathrm{\textnormal{e}}^{- \varphi})\right]\\
& = & (-1)^k  \sum_{i=0}^k  C_k^i\bar{\partial}^{k} \left[\bar{\partial}^{k-i}\phi (\mathrm{\textnormal{e}}^{ \varphi} \bar{\partial}^i  \mathrm{\textnormal{e}}^{- \varphi})\right]\\
& = & (-1)^k  \sum_{i,j=0}^k C_k^i C_k^j\left(\bar{\partial}^{k-j}\bar{\partial}^{k-i}\phi\right)\bar{\partial}^{j}(\mathrm{\textnormal{e}}^{ \varphi} \bar{\partial}^i  \mathrm{\textnormal{e}}^{- \varphi})\\
& = & (-1)^k  \sum_{i,j=0}^k O_{i,j,k}\left(\bar{\partial}^{k-j}\bar{\partial}^{k-i}\phi\right)\bar{\partial}^{j}(\mathrm{\textnormal{e}}^{ \varphi} \bar{\partial}^i  \mathrm{\textnormal{e}}^{- \varphi}).
\end{eqnarray*}
Then we have :
\begin{eqnarray*}
R ( \partial_\varphi^{k*} ( \phi) ) & = & \partial^k \left[  (-1)^k  \sum_{i,j=0}^k O_{i,j,k}\left(\bar{\partial}^{k-j}\bar{\partial}^{k-i}\phi\right)\bar{\partial}^{j}(\mathrm{\textnormal{e}}^{ \varphi} \bar{\partial}^i  \mathrm{\textnormal{e}}^{- \varphi})\right]  \\  
& = & (-1)^k  \sum_{i,j=0}^k  C_k^i  C_k^j \partial^k \left[ (\bar{\partial}^{k-j} \bar{\partial}^{k-i}\phi)  \bar{\partial}^j (\mathrm{\textnormal{e}}^{ \varphi} \bar{\partial}^i  \mathrm{\textnormal{e}}^{- \varphi})\right] \\
& = & (-1)^k  \sum_{i,j,l=0}^k  C_k^l O_{i,j,k} (\partial^{k-l}\bar{\partial}^{k-j} \bar{\partial}^{k-i}\phi) \partial^{l} \bar{\partial}^j (\mathrm{\textnormal{e}}^{ \varphi} \bar{\partial}^i  \mathrm{\textnormal{e}}^{- \varphi})\\
& = & (-1)^k  \sum_{i,j,l=0}^k  O_{i,j,l,k} (\partial^{k-l}\bar{\partial}^{k-j} \bar{\partial}^{k-i}\phi) \partial^{l} \bar{\partial}^j (\mathrm{\textnormal{e}}^{ \varphi} \bar{\partial}^i  \mathrm{\textnormal{e}}^{- \varphi}) .
 \end{eqnarray*}
Note that if :
$$
\partial_\varphi^{k*} ( \phi)  = (-1)^k \mathrm{\textnormal{e}}^{ \varphi}  \bar{\partial}^k ( \phi \mathrm{\textnormal{e}}^{- \varphi})\;\;\mbox{then}\;\;\partial_\varphi^{k*}(R  ( \phi) ) = (-1)^k \mathrm{\textnormal{e}}^{ \varphi}  \bar{\partial}^k ( R \phi  \mathrm{\textnormal{e}}^{- \varphi}) 
$$
with $\bar{\partial}^k ( R \phi  \mathrm{\textnormal{e}}^{- \varphi}) = \sum_{i=0}^k C_k^i    \bar{\partial}^{k-i}R \phi (  \bar{\partial}^i \mathrm{\textnormal{e}}^{- \varphi})$ . Therefore :

 $$ \partial_\varphi^{k*}(R  ( \phi) ) =  (-1)^k  \sum_{i=0}^k C_k^i    \bar{\partial}^{k-i}\partial^k \bar{\partial}^k \phi (\mathrm{\textnormal{e}}^{ \varphi}  \bar{\partial}^i \mathrm{\textnormal{e}}^{- \varphi})$$
So  
 $$R ( \partial_\varphi^{k*} ( \phi) ) - \partial_\varphi^{k*}(R  ( \phi) )  = (-1)^k  \sum_{i,j,l=1}^k  O_{i,j,l,k} (\partial^{k-l}\bar{\partial}^{k-j} \bar{\partial}^{k-i}\phi) \partial^{l} \bar{\partial}^j (\mathrm{\textnormal{e}}^{ \varphi} \bar{\partial}^i  \mathrm{\textnormal{e}}^{- \varphi})$$
Thus
 \begin{eqnarray*}
<\phi,  R( \bar{\partial}_\varphi^{k*}  \phi) - \bar{\partial}_\varphi^{k*} (  R \phi)>_\varphi & = & <\phi, (-1)^k \sum_{i,j,l=1}^k  O_{i,j,l,k} (\partial^{k-l}\partial^{k-i}\bar{\partial}^{k-j}\phi)\times \partial^l \bar{\partial}^j (\mathrm{\textnormal{e}}^{ \varphi} \partial^i  \mathrm{\textnormal{e}}^{- \varphi})> 
\end{eqnarray*}
\vskip 5mm
2)Applying the same technique as before, we have :
$$
R( \bar\partial_\varphi^{k*}  \phi) = (-1)^k  \sum_{i,j,l=0}^k  O_{i,j,l,k} (\partial^{k-l}\bar{\partial}^{k-j} \partial^{k-i}\phi) \partial^{l} \bar{\partial}^j (\mathrm{\textnormal{e}}^{ \varphi} \partial^i  \mathrm{\textnormal{e}}^{- \varphi})
$$
because
$$
\bar\partial_\varphi^{k*} ( \phi)  = (-1)^k \mathrm{\textnormal{e}}^{ \varphi} \partial^k ( \phi \mathrm{\textnormal{e}}^{- \varphi})\;\;\mbox{and}\;\;\partial^k ( \phi \mathrm{\textnormal{e}}^{- \varphi}) = \sum_{i=0}^k C_k^i \partial^{k-i}\phi (\partial^i \mathrm{\textnormal{e}}^{- \varphi}) 
$$
and
$$
\partial_\varphi^{k*}(R  ( \phi) ) =  (-1)^k  \sum_{i=0}^k C_k^i  \partial^{k-i}\partial^k \bar{\partial}^k \phi (\mathrm{\textnormal{e}}^{ \varphi}  \partial^i \mathrm{\textnormal{e}}^{- \varphi})
$$
because if :
$$
\bar\partial_\varphi^{k*} ( \phi)  = (-1)^k \mathrm{\textnormal{e}}^{ \varphi} \partial^k ( \phi \mathrm{\textnormal{e}}^{- \varphi})\;\;\mbox{then}\;\;\bar\partial_\varphi^{k*}(R  ( \phi) ) = (-1)^k \mathrm{\textnormal{e}}^{ \varphi}  \partial^k ( R \phi  \mathrm{\textnormal{e}}^{- \varphi})
$$
with $\partial^k ( R \phi \mathrm{\textnormal{e}}^{- \varphi}) = \sum_{i=0}^k C_k^i    \partial^{k-i}R \phi (\partial^i \mathrm{\textnormal{e}}^{- \varphi})$ . So
$$
R( \bar\partial_\varphi^{k*}  \phi) - \bar\partial_\varphi^{k*} (  R \phi) = (-1)^k  \sum_{i,j,l=1}^k  O_{i,j,l,k} (\partial^{k-l}\bar{\partial}^{k-j} \partial^{k-i}\phi) \partial^{l} \bar{\partial}^j (\mathrm{\textnormal{e}}^{ \varphi} \partial^i  \mathrm{\textnormal{e}}^{- \varphi})
$$
 Thus
$$
<\phi,  R( \bar\partial_\varphi^{k*}  \phi) - \bar\partial_\varphi^{k*} (  R \phi)>_\varphi  =  <\phi, (-1)^k \sum_{i,j,l=1}^k O_{i,j,l,k} (\partial^{k-l}\bar{\partial}^{k-j}\partial^{k-i}\phi)\times \partial^l \bar{\partial}^j (\mathrm{\textnormal{e}}^{ \varphi} \partial^i  \mathrm{\textnormal{e}}^{- \varphi})>_\varphi .
$$
\vskip 5mm
3)We have :
$$\partial^k (\bar{\partial}_\varphi^{k*}   \partial_\varphi^{k*} ( \phi) ) =   \sum_{i,j,l=0}^k  O_{i,j,l,k} (\partial^{k-l}\partial^{k-j} \bar{\partial}^{k-i}\phi)  \partial^{l}(\mathrm{\textnormal{e}}^{ \varphi}\partial^{j} \bar{\partial}^i  \mathrm{\textnormal{e}}^{- \varphi})$$
and
$$\bar{\partial}_\varphi^{k*}   \partial_\varphi^{k*} (\partial^k ( \phi) ) =  \sum_{i,j=0}^k  O_{i,j,k}  (\partial^{k-j} \bar{\partial}^{k-i}\partial^{k}\phi)  (\mathrm{\textnormal{e}}^{ \varphi}\partial^{j} \bar{\partial}^i  \mathrm{\textnormal{e}}^{- \varphi})$$
because :
$$
R^*\phi = \mathrm{\textnormal{e}}^{ \varphi}R(\phi\mathrm{\textnormal{e}}^{ -\varphi})\;\;\mbox{with}\;\;\bar\partial^{k}(\phi\mathrm{\textnormal{e}}^{ -\varphi}) = \sum_{i = 0}^k C_{k}^{i}\bar\partial^{k - i}\phi(\bar\partial^{i}\mathrm{\textnormal{e}}^{ -\varphi}).
$$
So
$$\partial^k (R^* ( \phi) ) - R^* (\partial^k ( \phi) )  =\sum_{i,j,l=1}^k  O_{i,j,l,k} (\partial^{k-l}\partial^{k-j} \bar{\partial}^{k-i}\phi)  \partial^{l}(\mathrm{\textnormal{e}}^{ \varphi}\partial^{j} \bar{\partial}^i  \mathrm{\textnormal{e}}^{- \varphi}).$$
Thus
\begin{eqnarray*}
<\phi,  \partial^k ( R^* \phi) - R^*  \phi)>_\varphi & = & <\phi,  \sum_{i,j,l=1}^k O_{i,j,l,k}  (\partial^{k-l}\bar\partial^{k-j}\partial^{k-i}\phi) \times  \partial^l (\mathrm{\textnormal{e}}^{ \varphi} \partial^i \bar{\partial}^j  \mathrm{\textnormal{e}}^{- \varphi})>_\varphi .
\end{eqnarray*}
\vskip 5mm
4)We deduce from the previous result the following result 
\begin{eqnarray*}
<\phi,  R^*\partial_\varphi^{k*} \phi) - R^*(  \bar{\partial}^k  \phi)>_\varphi & = & <\phi,  \sum_{i,j,l=1}^k O_{i,j,l,k}  (\bar{\partial}^{k-l}\partial^{k-i}\bar{\partial}^{k-j}\phi) \times \bar{\partial}^l (\mathrm{\textnormal{e}}^{ \varphi} \partial^i \bar{\partial}^j  \mathrm{\textnormal{e}}^{- \varphi})>_\varphi 
\end{eqnarray*}
using the same computation techniques.
\vskip 5mm
5)Not that 
\begin{eqnarray*}
\bar{ \partial}_\varphi^{k*} ( \phi) & = &(-1)^k \mathrm{\textnormal{e}}^{ \varphi}  \partial^k ( \phi \mathrm{\textnormal{e}}^{- \varphi})\\ 
 & = &(-1)^k  \sum_{i=0}^k C_k^i    \partial^{k-i}\phi (\mathrm{\textnormal{e}}^{ \varphi}  \partial^i \mathrm{\textnormal{e}}^{- \varphi})  
 \end{eqnarray*}
 \begin{eqnarray*}
\Longrightarrow \bar{\partial}^k ( \partial_\varphi^{k*} ( \phi) )& = &(-1)^k \sum_{i=0}^k C_k^i   \bar{\partial}^k [ \bar{\partial}^{k-i}\phi (\mathrm{\textnormal{e}}^{ \varphi}  \bar{\partial}^i \mathrm{\textnormal{e}}^{- \varphi}) ]\\ 
& = & (-1)^k  \sum_{i,j=0}^k  O_{i,j,l,k}  (\bar{\partial}^{k-j} \bar{\partial}^{k-i}\phi)  \bar{\partial}^j (\mathrm{\textnormal{e}}^{ \varphi} \bar{\partial}^i  \mathrm{\textnormal{e}}^{- \varphi})
  \end{eqnarray*}
  and
  $$\partial_\varphi^{k*} (\bar{\partial}^k ( \phi) ) = (-1)^k  \sum_{i=0}^k  C_k^i  (\bar{\partial}^{k-i} \bar{\partial}^{k}\phi) (\mathrm{\textnormal{e}}^{ \varphi} \bar{\partial}^i  \mathrm{\textnormal{e}}^{- \varphi})$$
  so
  $$\bar{\partial}^k ( \partial_\varphi^{k*} ( \phi) ) -\partial_\varphi^{k*} (\bar{\partial}^k ( \phi) ) = (-1)^k  \sum_{i,j=1}^k  O_{i,j,k}  (\bar{\partial}^{k-j} \bar{\partial}^{k-i}\phi)  \bar{\partial}^j (\mathrm{\textnormal{e}}^{ \varphi} \bar{\partial}^i  \mathrm{\textnormal{e}}^{- \varphi}).$$ 
  Thus
  \begin{eqnarray*}
  <\phi,  \bar{\partial}^k ( \partial_\varphi^{k*} \phi) - \partial_\varphi^{k*} (  \bar{\partial}^k  \phi)>_\varphi & = & <\phi, (-1)^k  \sum_{i,j=1}^k  O_{i,j,k} (\bar{\partial}^{k-j} \bar{\partial}^{k-i}\phi) \times \bar{\partial}^j (\mathrm{\textnormal{e}}^{ \varphi} \bar{\partial}^i  \mathrm{\textnormal{e}}^{- \varphi})>_\varphi .
  \end{eqnarray*}
\vskip 5mm
6) We deduce from the previous result that :
   \begin{eqnarray*}
  <\phi,  \partial^k ( \bar{\partial}_\varphi^{k*} \phi) - \bar{\partial}_\varphi^{k*} (  \partial^k  \phi)>_\varphi & = & <\phi, (-1)^k  \sum_{i,j=1}^k  O_{i,j,k} (\partial^{k-j} \partial^{k-i}\phi) \times \partial^j (\mathrm{\textnormal{e}}^{ \varphi} \partial^i  \mathrm{\textnormal{e}}^{- \varphi})>_\varphi 
  \end{eqnarray*}
\end{proof}
\begin{lemma} \label{E}
Let $\varphi = \vert z \vert ^2$ then
\begin{eqnarray*}
1) <\phi,  R( \bar{\partial}_\varphi^{k*}  \phi) - \bar{\partial}_\varphi^{k*} (  R \phi)>_\varphi = 0\\
2) <\phi,  R( \partial_\varphi^{k*}  \phi) - \partial_\varphi^{k*} (  R \phi)>_\varphi = 0\\
3)\;\;\; <\phi,  \partial^k ( R\phi) - R (  \partial^k  \phi)>_\varphi =0\\
4) <\phi,  \bar{\partial}^k ( \partial_\varphi^{k*} \phi) - \partial_\varphi^{k*} (  \bar{\partial}^k  \phi)>_\varphi =0\\ 
5) <\phi, \partial^k( R^* \phi) - R^*(  \partial^k  \phi)>  & = & <\phi,  \sum_{i=1}^k \sum_{n=0}^i \sum_{j=i-n}^k  \sum_{l=1}^{j-i+n} 
O_{i,j,l,n,k} (\partial^{k-l} \partial^{k-i} 
\bar{\partial}^{k-j} \phi)\times \\&& (-1)^{j+n}    
\frac{j!}{(j-i+n-l)!}\times(z)^{j-i+n-l} (\bar{z})^{n}>\\
6) <\phi, \bar{\partial}^k( R^* \phi) - R^*(  \bar{\partial}^k  \phi)>  & = & <\phi,  \sum_{i=1}^k \sum_{n=0}^i \sum_{j=i-n}^k  \sum_{l=1}^{j-i+n} 
O_{i,j,l,n,k} \times (\bar{\partial}^{k-l} \partial^{k-i} 
\bar{\partial}^{k-j} \phi)\times \\& & (-1)^{j+n}    
\frac{j!}{(j-i+n)!}\frac{n!}{(n-l)!}\times (z)^{j-i+n} (\bar{z})^{n-l)}>\\
  \end{eqnarray*}
  with
$$
C_k^i  C_k^j C_k^l C_{k}^{n} = O_{i,j,l,n,k} .
$$
\end{lemma}
\begin{proof}
Let $h(g) = \mathrm{\textnormal{e}}^{g}$ with $g = - \varphi$, according to the formula of Fa\`{a} di Bruno in \cite{2}, we have
 \begin{eqnarray*}
\bar{\partial}^i \mathrm{\textnormal{e}}^{g} & = & \bar{\partial}^i(h(g))\\
 & = &  \sum \frac{i!}{m_1!m_2! \cdots m_i!}h^{(m_1 + \cdots + m_i)}(g)\prod_{\gamma=1}^i (\frac{\bar{\partial}^\gamma g}{\gamma !})^{m_\gamma}\\ 
 & = &  (\sum \frac{i!}{m_1!m_2! \cdots m_i!} \prod_{\gamma=1}^i (\frac{-\bar{\partial}^\gamma \varphi}{\gamma !})^{m_\gamma})\mathrm{\textnormal{e}}^{- \varphi}\\
 & =: & P_i \mathrm{\textnormal{e}}^{- \varphi}\\
 \end{eqnarray*}
 where the sum runs over all $i$-tuples of positive integers $(m_1, m_2, \cdots , m_i)$ satisfying the constraint 
$$
1m_1 + 2m_2+ \cdots + im_i = i
$$.
 Let $\varphi = \vert z \vert ^2$, we have
 \[ 
 \bar{\partial}^\gamma \varphi = \left \{
 \begin{array}{rl}
 z & \mbox{ if } \gamma = 1\\
 0 & \mbox{ if } \gamma \geq 2
 \end{array}
 \right.
 \]
 $$P_i = (-\bar{\partial}^\gamma \varphi)^i = (-z)^i = (-1)^i (z)^i$$
 $$ \bar{\partial}^i \mathrm{\textnormal{e}}^{-\vert z \vert ^2} =  (-1)^i (z)^i \mathrm{\textnormal{e}}^{-\vert z \vert ^2}.$$
We begin in this proof to show item 4). We thus have previous computations :\\
\vskip 0.25cm
 \begin{equation*} \label{D1}
   \mathrm{\textnormal{e}}^{\vert z \vert ^2}\bar{\partial}^i \mathrm{\textnormal{e}}^{-\vert z \vert ^2} = (-1)^i (z)^i
  \end{equation*}
  \begin{equation*} \label{D2}
  \bar{\partial}^j ( \mathrm{\textnormal{e}}^{\vert z \vert ^2}\bar{\partial}^i \mathrm{\textnormal{e}}^{-\vert z \vert ^2}) = 0
  \end{equation*}
  thus
  $$<\phi,  \bar{\partial}^k ( \partial_\varphi^{k*} \phi) - \partial_\varphi^{k*} (  \bar{\partial}^k  \phi)>_\varphi =0.$$
\vskip 0.5mm
3)  We have
  $$ \partial^i \mathrm{\textnormal{e}}^{-\vert z \vert ^2} =  (-1)^i (\bar{z})^i \mathrm{\textnormal{e}}^{-\vert z \vert ^2}$$
 Thus
 \begin{equation*} \label{D3}
   \mathrm{\textnormal{e}}^{\vert z \vert ^2}\partial^i \mathrm{\textnormal{e}}^{-\vert z \vert ^2} = (-1)^i (\bar{z})^i
  \end{equation*}
  \begin{equation*} \label{D4}
  \partial^j ( \mathrm{\textnormal{e}}^{\vert z \vert ^2}\partial^i \mathrm{\textnormal{e}}^{-\vert z \vert ^2}) = 0
  \end{equation*}
  thus
  $$<\phi,  \partial^k ( \partial_\varphi^{k*} \phi) - \partial_\varphi^{k*} ( \partial^k  \phi)>_\varphi = 0.$$\\
\vskip 0.5mm  
1)  We have
   \[ 
\bar{ \partial}^{j} (\mathrm{\textnormal{e}}^{\vert z \vert ^2}\partial^i \mathrm{\textnormal{e}}^{-\vert z \vert ^2}) = \left \{
 \begin{array}{lr}
 0 & \mbox{ if } i < j \\ & \\
(-1)^i \frac{i!}{(i-j)!} (\bar{z})^{i-j} & \mbox{ if } i \geq j
 \end{array}
 \right.
 \]
 so
 $$\partial^l\bar{ \partial}^{j} (\mathrm{\textnormal{e}}^{\vert z \vert ^2}\partial^i \mathrm{\textnormal{e}}^{-\vert z \vert ^2}) = 0$$
 thus
 $$<\phi,  R( \bar{\partial}_\varphi^{k*}  \phi) - \bar{\partial}_\varphi^{k*} (  R\phi)>_\varphi  =0.$$
 \, \\
 \vskip 0.5mm
We obtain in a similar way item 2).
\vskip 0.5mm
5) From the following relation 
 \begin{equation} \label{D4}
  \partial^i \bar{\partial}^j\mathrm{\textnormal{e}}^{-\vert z \vert ^2} = (-1)^j \partial^i((z)^j \mathrm{\textnormal{e}}^{-\vert z \vert ^2})
  \end{equation}
we obtain by explaining the writing of the derivative of the product of two functions, this new relation
  \begin{equation} \label{D5}
  \partial^i \bar{\partial}^j \mathrm{\textnormal{e}}^{-\vert z \vert ^2} = (-1)^j \sum_{n=0}^i  C_i^n  \partial^{i-n}((z)^j)\partial^n \mathrm{\textnormal{e}}^{-\vert z \vert ^2}
  \end{equation}
to which we apply the formula of Fa\`{a} di Bruno and we therefore obtain :
 $$P_j  = (-1)^j (z)^j$$
 $$ \partial^n \mathrm{\textnormal{e}}^{-\vert z \vert ^2} =  (-1)^n (\bar{z})^n \mathrm{\textnormal{e}}^{-\vert z \vert ^2}.$$
 Substituting in \ref{D5}, we have
 \begin{equation} \label{E1}
 \partial^i \bar{\partial}^j \mathrm{\textnormal{e}}^{-\vert z \vert ^2} = (-1)^j \sum_{n=0}^i (-1)^n C_i^n  \partial^{i-n}(z)^j (\bar{z})^n \mathrm{\textnormal{e}}^{-\vert z \vert ^2}
 \end{equation}
 or
 \[ 
 \partial^{i-n} (z)^j = \left \{
 \begin{array}{lr}
 0 & \mbox{ if } j < i-n \\ & \\
 \frac{j!}{(j-i+n)!} (z)^{j-i+n} & \mbox{ if } j \geq i-n
 \end{array}
 \right.
 \]
 By replacing new in \ref{E1}, we have
 \[ 
 \partial^i \bar{\partial}^j \mathrm{\textnormal{e}}^{-\vert z \vert ^2} = \left \{
 \begin{array}{lr}
 0 & \mbox{ if } j < i-n \\ & \\
 \sum_{n=0}^i (-1)^{n+j} C_i^n  \frac{j!}{(j-i+n)!}  (z)^{j-i+n} (\bar{z})^n \mathrm{\textnormal{e}}^{-\vert z \vert ^2} & \mbox{ if } j \geq i-n
 \end{array}
 \right.
 \]
 \;\\
 
  \[ 
 \mathrm{\textnormal{e}}^{\vert z \vert ^2}\partial^i \bar{\partial}^j \mathrm{\textnormal{e}}^{-\vert z \vert ^2} = \left \{
 \begin{array}{lr}
 0 & \mbox{ if } j < i-n \\ & \\
 \sum_{n=0}^i (-1)^{n+j} C_i^n  \frac{j!}{(j-i+n)!}  (z)^{j-i+n} (\bar{z})^n  & \mbox{ if } j \geq i-n
 \end{array}
 \right.
 \]
we have
$$\partial^l(\mathrm{\textnormal{e}}^{\vert z \vert ^2} \partial^i \bar{\partial}^j \mathrm{\textnormal{e}}^{-\vert z \vert ^2}) =  \sum_{n=0}^i (-1)^{n+j} C_i^n  \frac{j!}{(j-i+n)!} \partial^l (z)^{j-i+n} (\bar{z})^n   \mbox{ if } j \geq i-n $$
\;\\
\[ 
\partial^l( \mathrm{\textnormal{e}}^{\vert z \vert ^2} \partial^i \bar{\partial}^j \mathrm{\textnormal{e}}^{-\vert z \vert ^2}) = \left \{
 \begin{array}{lr}
 0  & \\ \mbox{ if } j-i  +n < l \mbox{ and } j < i-n & \\& \\& \\
 \sum_{n=0}^i (-1)^{n+j} C_i^n  \frac{j!}{(j-i+n-l)!}  (z)^{i-j+n-l} (\bar{z})^{n}  & \\ & \\ \mbox{ if } j \geq i-n \mbox{ and } j-i+n \geq l
 \end{array}
 \right.
 \]
 \, \\
 \begin{eqnarray*} 
<\phi, \partial^k( R^* \phi) - R^*(  \partial^k  \phi)>  & = & <\phi,  \sum_{i=1}^k \sum_{n=0}^i \sum_{j=i-n}^k  \sum_{l=1}^{j-i+n} 
O_{i,j,l,n,k} (\partial^{k-l} \partial^{k-i} 
\bar{\partial}^{k-j} \phi)\times \\&& (-1)^{j+n}    
\frac{j!}{(j-i+n-l)!}\times(z)^{j-i+n-l} (\bar{z})^{n}>\\ 
\end{eqnarray*} 
Analogously, we have 
\begin{eqnarray*}
<\phi, \bar{\partial}^k( R^* \phi) - R^*(  \bar{\partial}^k  \phi)>  & = & <\phi,  \sum_{i=1}^k \sum_{n=0}^i \sum_{j=i-n}^k  \sum_{l=1}^{j-i+n} 
O_{i,j,l,n,k} \times (\bar{\partial}^{k-l} \partial^{k-i} 
\bar{\partial}^{k-j} \phi)\times \\& & (-1)^{j+n}    
\frac{j!}{(j-i+n)!}\frac{n!}{(n-l)!}\times (z)^{j-i+n} (\bar{z})^{n-l}>\\ 
  \end{eqnarray*}
\end{proof}
\begin{lemma} \label{F}
  Let $\varphi = \vert z \vert ^2$ then
\begin{eqnarray*}
1)\; \alpha^2 <\phi, R( R^* \phi) - R^*( R \phi)>_\varphi & = & \alpha^2 \sum_{i =0}^{k-1} \sum_{j =0}^{k-1}\frac{(k!)^4}{(i!)^2(j!)^2((k-j)!)((k-i)!)} \times  \vert \vert \bar{\partial}^{j} \partial^{i} \phi \vert \vert_\varphi^2\\
2)\; \beta^2 <\phi, \bar{\partial}^{k}( \bar{\partial}^{k*} \phi) - \bar{\partial}^{k*} ( \bar{\partial}^{k} \phi) >_\varphi  
& = & \beta^2 \sum_{j =0}^{k-1} \frac{(k!)^2}{(j!)^2((k-j)!)}\times  \vert \vert \bar{\partial}^{j}  \phi \vert \vert_\varphi^2\\
3)\; \gamma^2 <\phi, \partial^{k}( \partial^{k*} \phi) - \partial^{k*} ( \partial^{k} \phi) >_\varphi  & = & \gamma^2 \sum_{i =0}^{k-1} \frac{(k!)^2}{(i!)^2((k-i)!)} \times\vert \vert \partial^{i}  \phi \vert \vert_\varphi^2 \\
4)\;\alpha \gamma <\phi,  \partial^k ( R^* \phi) - R^*(  \partial^k  \phi)>_\varphi & = &  \alpha  \gamma \sum_{l =1}^{j-i+n} \frac{(k!)^2}{((l)!)[(k+l)!]^2} \times
 < \bar{\partial}^{k} (\partial^{k-l} \phi) ,   \partial^{k-l} \phi>_\varphi \\
5)\;\alpha \beta <\phi,  \bar{\partial}^k ( R^* \phi) - R^*(  \bar{\partial}^k  \phi)>_\varphi & = &  \alpha  \beta \sum_{l =1}^{k} \frac{(k!)^2}{((l)!)[(k+l)!]^2} \times
 < \partial^{k} (\bar{\partial}^{k-l} \phi) ,   \bar{\partial}^{k-l} \phi>_\varphi  .
\end{eqnarray*}
 \end{lemma}
 \begin{proof}
\vskip 0.5mm
According to Lemma $2.5$ of \cite{3}, we have :
 \begin{eqnarray*}
\alpha^2 <\phi, R( R^* \phi) - R^*(  R \phi)>_\varphi & = & \alpha^2 \sum_{i =0}^{k-1} \sum_{j =0}^{k-1}\frac{(k!)^4}{(i)^2(j)^2((k-j)!)((k-i)!)} \times  \vert \vert \bar{\partial}^{j} \partial^{i} \phi \vert \vert_\varphi^2 .
\end{eqnarray*}
\vskip 0.5mm
2) This item is obtained from the Lemma $2.4$ of \cite{1}. Thus :
\begin{eqnarray*}
\beta^2 <\phi, \bar{\partial}^{k}( \bar{\partial}^{k*} \phi) - \bar{\partial}^{k*} ( \bar{\partial}^{k} \phi) >_\varphi  
& = & \beta^2 \sum_{j =0}^{k-1} \frac{(k!)^2}{(j!)^2((k-j)!)} \times \vert \vert \bar{\partial}^{j}  \phi \vert \vert_\varphi^2 .
\end{eqnarray*}
\vskip 0.5mm
3)By conjugation, we have
\begin{eqnarray*}
 \gamma^2 <\phi, \partial^{k}( \partial^{k*} \phi) - \partial^{k*} ( \partial^{k} \phi) >_\varphi  & = & \gamma^2 \sum_{i =0}^{k-1} \frac{(k!)^2}{(i!)^2((k-i)!)} \times  \vert \vert \partial^{i}  \phi \vert \vert_\varphi^2 .
\end{eqnarray*}
\vskip 0.5mm
5) According to Lemma \ref{E},
 \begin{eqnarray*}
<\phi, \bar{\partial}^k( R^* \phi) - R^*(  \bar{\partial}^k  \phi)>  & = & <\phi,  \sum_{i=1}^k \sum_{n=0}^i \sum_{j=i-n}^k  \sum_{l=1}^{j-i+n} 
O_{i,j,l,n,k} \times(\bar{\partial}^{k-l} \partial^{k-i} 
\bar{\partial}^{k-j} \phi)\times\\ &&(-1)^{j+n}    
\frac{j!}{(j-i+n)!}\frac{n!}{(n-l)!}\times (z)^{j-i+n} (\bar{z})^{n-l}>\\ .
  \end{eqnarray*}
  
  Let $s = n-l$
  \begin{eqnarray*}
<\phi, \bar{\partial}^k( R^* \phi) - R^*(  \bar{\partial}^k  \phi)>  & = & <\phi,\sum_{i=1}^k  \sum_{j=i-n}^k  \sum_{l=1}^{k} 
A  (\bar{\partial}^{k-l} \partial^{k-i} 
\bar{\partial}^{k-j} \phi) \times \sum_{s=0}^{i-l} (-1)^{j+l+s}    C_{i-l}^s \times\\&& \frac{j!}{(j-i+l+s)!}\times (z)^{j-i+l+s} (\bar{z})^{s}>
 \end{eqnarray*}
 with
 $$A=A(k,j,l,i) := \frac{(k!)^3}{l!j!(k-i)!(k-j)! (k-l)!(i-l)!}$$
 we have
  \begin{eqnarray*}
<\phi, \bar{\partial}^k( R^* \phi) - R^*(  \bar{\partial}^k  \phi)>  & = & <\phi,\sum_{i=1}^k  \sum_{j=i-n}^k  \sum_{l=1}^{k} (-1)^{l}
A  (\bar{\partial}^{k-l} \partial^{k-i} 
\bar{\partial}^{k-j} \phi) \mathrm{\textnormal{e}}^{\vert z \vert ^2} \partial^{i-l}\bar{\partial}^{j} \mathrm{\textnormal{e}}^{-\vert z \vert ^2}>
 \end{eqnarray*}
 \begin{align*}
[\bar{\partial}^k( R^* \phi) - R^*(  \bar{\partial}^k  \phi)] \mathrm{\textnormal{e}}^{-\vert z \vert ^2} = & \sum_{i=1}^k  \sum_{j=i-n}^k  \sum_{l=1}^{k} (-1)^{l}
A  (\bar{\partial}^{k-l} \partial^{k-i} 
\bar{\partial}^{k-j} \phi) \partial^{i-l}\bar{\partial}^{j} \mathrm{\textnormal{e}}^{-\vert z \vert ^2}&
\end{align*}
 Let $s = i-l$
 \begin{align*}
[\bar{\partial}^k( R^* \phi) - R^*(  \bar{\partial}^k  \phi)] \mathrm{\textnormal{e}}^{-\vert z \vert ^2} = & \sum_{l=1}^k  \sum_{j=i-n}^k   (-1)^{l}
A'  \sum_{s=0}^{k-l} C_{k-l}^s(\bar{\partial}^{k-l} \partial^{k-l-s} 
\bar{\partial}^{k-j} \phi) \partial^{s}\bar{\partial}^{j} \mathrm{\textnormal{e}}^{-\vert z \vert ^2}&
\end{align*}
with
 $$A'=A'(k,j,l,m,i) := \frac{(k!)^3}{l!j!(k-j)! [(k-l)!]^2}$$
 so we have 
\begin{align*}
[\bar{\partial}^k( R^* \phi) - R^*(  \bar{\partial}^k  \phi)] \mathrm{\textnormal{e}}^{-\vert z \vert ^2} = & \sum_{l=1}^k  \sum_{j=i-n}^k   (-1)^{l}
A' \partial^{k-l} (\bar{\partial}^{k-l}  
\bar{\partial}^{k-j} \phi \bar{\partial}^{j} \mathrm{\textnormal{e}}^{-\vert z \vert ^2})&
\end{align*} 
and
\begin{align*}
[\bar{\partial}^k( R^* \phi) - R^*(  \bar{\partial}^k  \phi)] \mathrm{\textnormal{e}}^{-\vert z \vert ^2} = & \sum_{l=1}^k A'' \sum_{j=0}^k   (-1)^{l}C_{k}^j
 \partial^{k-l} (\bar{\partial}^{k-l}  
\bar{\partial}^{k-j} \phi \bar{\partial}^{j} \mathrm{\textnormal{e}}^{-\vert z \vert ^2})&
\end{align*} 
 with
 $$A''=A''(k,j,l,m,i) := \frac{(k!)^2}{l! [(k-l)!]^2}$$
 so
 \begin{align*}
[\bar{\partial}^k( R^* \phi) - R^*(  \bar{\partial}^k  \phi)] \mathrm{\textnormal{e}}^{-\vert z \vert ^2} = & \sum_{l=1}^k A''   (-1)^{l}
 \partial^{k-l} \bar{\partial}^{k}(  
\bar{\partial}^{k-l} \phi \mathrm{\textnormal{e}}^{-\vert z \vert ^2})& 
\end{align*}
Coming back to the computation with weighted inner product  , we have,
\begin{eqnarray*}
<\phi, \bar{\partial}^k( R^* \phi) - R^*(  \bar{\partial}^k  \phi)>_\varphi & = & \int_\mathbb{C} \overline{\phi}\bar{\partial}^k ( R^* \phi) \mathrm{\textnormal{e}}^{-\vert z \vert ^2} d\sigma - \\ & & \int_\mathbb{C} \overline{\phi}R^*(  \bar{\partial}^k  \phi)\mathrm{\textnormal{e}}^{-\vert z \vert ^2} d\sigma\\
  & = & \sum_{l=1}^k  A''(-1)^{l} \times\\ & & \int_\mathbb{C} \overline{\phi} \partial^{k-l} \bar{\partial}^{k}(  
\bar{\partial}^{k-l} \phi \mathrm{\textnormal{e}}^{-\vert z \vert ^2}) d\sigma\\
  & = &  \sum_{l=1}^k A''(-1)^{k} \times\\ & & \int_\mathbb{C}\partial^{k-l}  \overline{\phi} \bar{\partial}^{k}(  
\bar{\partial}^{k-l} \phi \mathrm{\textnormal{e}}^{-\vert z \vert ^2}) d\sigma\\
  & = & \sum_{l=1}^k A'' \times\\ & & \int_\mathbb{C}\bar{\partial}^{k} \partial^{k-l}  \overline{\phi} \bar{\partial}^{k-l}  \phi   \mathrm{\textnormal{e}}^{-\vert z \vert ^2} d\sigma\\ 
& = &  \sum_{l=1}^k A'' \times\\ & & \int_\mathbb{C} \overline{\partial^{k} \bar{\partial}^{k-l} \phi}   \bar{\partial}^{k-l}  \phi    \mathrm{\textnormal{e}}^{-\vert z \vert ^2} d\sigma\\
  & = &  \sum_{l=1}^k A''   <\partial^{k} \bar{\partial}^{k-l} \phi,  \bar{\partial}^{k-l}  \phi>_\varphi
  \end{eqnarray*}
   with
 $$A''=A''(k,j,l,m,i) := \frac{(k!)^2}{l! [(k-l)!]^2}$$
 so
  \begin{eqnarray*}
\alpha \beta   <\phi, \bar{\partial}^k( R^* \phi) - R^*(  \bar{\partial}^k  \phi)>_\varphi & = &  \alpha  \beta \sum_{l=1}^{k} \frac{(k!)^2}{l! [(k-l)!]^2}   \times <\partial^{k}(\bar{\partial}^{k-l} \phi),   \bar{\partial}^{k-l}  \phi >_\varphi 
  \end{eqnarray*}
  By analogy we have\\
\vskip 0.5mm
6)   \begin{eqnarray*}
\alpha \gamma   <\phi, \partial^k( R^* \phi) - R^*(  \partial^k  \phi)>_\varphi & = &  \alpha  \gamma \sum_{l=1}^{k} \frac{(k!)^2}{l! [(k-l)!]^2}   \times <\bar{\partial}^{k}(\partial^{k-l} \phi),   \partial^{k-l}  \phi >_\varphi
  \end{eqnarray*}
 \end{proof}
\section{Proof of the main Theorem}
\begin{proof}{(Theorem \ref{P})}
Let $\varphi = \vert z \vert^2$ and $\phi \in C_0^\infty(\mathbb{C})$. According to Lemma \ref{B}
\begin{eqnarray*}
\vert \vert H_\varphi^* \phi \vert \vert_\varphi^2 & \geq & \alpha^2 <\phi,R( R^* \phi) - R^*(R \phi) >_\varphi +  \beta^2 <\phi, \bar{\partial}^{k}( \bar{\partial}^{k*} \phi) - \bar{\partial}^{k*} ( \bar{\partial}^{k} \phi) >_\varphi \\  & + & \gamma^2 <\phi, \partial^{k}( \partial^{k*} \phi) - \partial^{k*} ( \partial^{k} \phi) >_\varphi + \alpha \beta <\phi,R( \bar{\partial}^{k*} \phi) +\bar{\partial}^{k}( R^* \phi)  - R^*( \bar{\partial}^{k} \phi) \\ & - & \bar{\partial}^{k*}( R \phi) >_\varphi  + \alpha \gamma <\phi,R( \partial^{k*} \phi) +\partial^{k}( R^* \phi)  - R^*( \partial^{k} \phi) - \partial^{k*}( R \phi) >_\varphi \\  & + &  \beta \gamma <\phi, \bar{\partial}^{k}( \partial^{k*} \phi) +\partial^{k}( \bar{\partial}^{k*}  \phi)  - \bar{\partial}^{k*} ( \partial^{k} \phi) - \partial^{k*}( \bar{\partial}^{k} \phi) >_\varphi 
\end{eqnarray*}
by setting 
$$
H_\varphi =\alpha \partial^k \bar{\partial}^{k} + \beta \bar{\partial}^k +\gamma  \partial^k  + c .
$$
According to the Lemma \ref{C}
\begin{eqnarray*}
 \alpha^2 <\phi,R( R^* \phi) - R^*(R \phi) >_\varphi \\  +  \beta^2 <\phi, \bar{\partial}^{k}( \bar{\partial}^{k*} \phi) - \bar{\partial}^{k*} ( \bar{\partial}^{k} \phi) >_\varphi \\   + \gamma^2 <\phi, \partial^{k}( \partial^{k*} \phi) - \partial^{k*} ( \partial^{k} \phi) >_\varphi  & = & \alpha^2 \sum_{i =0}^{k-1} \sum_{j =0}^{k-1}\frac{(k!)^4}{(i!)^2(j!)^2((k-j)!)((k-i)!)} \times\\ & &  \vert \vert \bar{\partial}^{j} \partial^{i} \phi \vert \vert_\varphi^2 +  \beta^2 \sum_{j =0}^{k-1} \frac{(k!)^2}{(j!)^2((k-j)!)} \times\\ & &  \vert \vert \bar{\partial}^{j}  \phi \vert \vert_\varphi^2 + \gamma^2 \sum_{i =0}^{k-1} \frac{(k!)^2}{(i!)^2((k-i)!)} \times\\ & &  \vert \vert \partial^{i}  \phi \vert \vert_\varphi^2 \\
 \end{eqnarray*}
 so
 $$\vert \vert H_\varphi^* \phi \vert \vert_\varphi^2 \geq (\alpha^2(k!)^2 + \beta^2(k!)+  \gamma^2(k!))\vert \vert \phi \vert \vert_\varphi^2.$$
 According to the Cauchy-Schwarz inequality, we have
 \begin{eqnarray*}
\vert<f, \phi>_\varphi \vert^2 & \leq & \; \vert \vert f \vert \vert_\varphi^2 \vert \vert \phi \vert \vert_\varphi^2\\
& = &\frac{1}{\alpha^2(k!)^2 + \beta^2(k!)+  \gamma^2(k!)} \vert \vert f \vert \vert_\varphi^2 (\alpha^2(k!)^2 + \beta^2(k!)+  \gamma^2(k!)) \vert \vert \phi \vert \vert_\varphi^2\\
& \leq & \frac{1}{\alpha^2(k!)^2 + \beta^2(k!)+  \gamma^2(k!)} \vert \vert f \vert \vert_\varphi^2 \vert \vert (\alpha \partial^k \bar{\partial}^{k} + \beta \bar{\partial}^k +\gamma  \partial^k  + c)_\varphi^* \phi \vert \vert_\varphi^2.\\ 
 \end{eqnarray*}
 According to Lemma \ref{A}, there exists $u \in  L^2(\mathbb{\C}, \mathrm{\textnormal{e}}^{- \vert z \vert^2})$ such that 
 $$ H_\varphi u = f$$ with the estimation of the norm
 $$\vert \vert u \vert \vert_\varphi^2 \leq \frac{1}{\alpha^2(k!)^2 + \beta^2(k!)+  \gamma^2(k!)} \vert \vert f \vert \vert_\varphi^2$$
 that is to say
 $$\int_\mathbb{C} \vert u \vert^2 \mathrm{\textnormal{e}}^{- \vert z \vert^2} d\sigma \leq \frac{1}{\alpha^2(k!)^2 + \beta^2(k!)+  \gamma^2(k!)} \int_\mathbb{C} \vert f \vert^2 \mathrm{\textnormal{e}}^{- \vert z \vert^2} d\sigma.$$
\end{proof}\\
As a consequence we have
\begin{theorem}
There is a linear and bounded operator $$T_k : L^2(\mathbb{\C}, \mathrm{\textnormal{e}}^{- \vert z \vert^2}) \longrightarrow L^2(\mathbb{\C}, \mathrm{\textnormal{e}}^{- \vert z \vert^2})$$
such that $$ (\alpha \partial^k \bar{\partial}^{k} + \beta \bar{\partial}^k +\gamma  \partial^k  + c) T_k = I$$ with the estimation of the norm
$$\vert \vert T_k \vert \vert^2_\varphi \leq \frac{1}{\alpha^2(k!)^2 + \beta^2(k!)+  \gamma^2(k!)}$$ where $\vert \vert T_k \vert \vert_\varphi$ is the norm of $T_k$ in $L^2(\mathbb{\C}, \mathrm{\textnormal{e}}^{- \vert z \vert^2})$.
\end{theorem}
\begin{proof}
Let $f \in L^2(\mathbb{\C}, \mathrm{\textnormal{e}}^{- \vert z \vert^2})$. According to Theorem \ref{P}, there exists a weak solution $u \in L^2(\mathbb{\C}, \mathrm{\textnormal{e}}^{- \vert z \vert^2})$ such that
$$ \alpha \partial^k \bar{\partial}^{k}u + \beta \bar{\partial}^ku +\gamma  \partial^k u + cu = f$$ with the estimation of the norm 
 $$\vert \vert u \vert \vert^2_\varphi \leq \frac{1}{\alpha^2(k!)^2 + \beta^2(k!)+  \gamma^2(k!)} \vert \vert f \vert \vert^2_\varphi$$
so
$$ (\alpha \partial^k \bar{\partial}^{k} + \beta \bar{\partial}^k +\gamma  \partial^k  + c)T_k(f) = f$$ with
 $$\vert \vert T_k(f) \vert \vert^2_\varphi \leq \frac{1}{\alpha^2(k!)^2 + \beta^2(k!)+  \gamma^2(k!)} \vert \vert f \vert \vert^2_\varphi.$$
 Thus
 $$(\alpha \partial^k \bar{\partial}^{k} + \beta \bar{\partial}^k +\gamma  \partial^k  + c)T_k(f) = I$$ with 
 $$\vert \vert T_k \vert \vert^2_\varphi \leq \frac{1}{\alpha^2(k!)^2 + \beta^2(k!)+  \gamma^2(k!)}.$$
\end{proof}

\section{The different cases}
\subsection{Case where o\`u $\alpha= \gamma=0$:  The operator $\beta \bar{\partial}^{k}  + c$} 
%\;\\

Taking into account Theorem $1.1$ of \cite{1}, we have the following result : 
\begin{thm} \label{q}
For any $f \in L^2(\mathbb{\C}, \mathrm{\textnormal{e}}^{- \vert z \vert^2})$, there exists a weak solution $u \in L^2(\mathbb{\C}, \mathrm{\textnormal{e}}^{- \vert z \vert^2})$ of the equation
$$\beta \bar{\partial}^{k} u + cu = \beta f$$ with $\vert \beta \vert \geq 1$ and the estimation of the norm
$$\int_\mathbb{C} \vert u \vert^2 \mathrm{\textnormal{e}}^{- \vert z \vert^2} d\sigma \leq \frac{\vert \beta \vert^2}{(k!)} \int_\mathbb{C} \vert f \vert^2 \mathrm{\textnormal{e}}^{- \vert z \vert^2} d\sigma$$
\end{thm}
\begin{proof}
Let $f \in L^2(\mathbb{\C}, \mathrm{\textnormal{e}}^{- \vert z \vert^2})$, according to Theoreme $1.1$ of \cite{1}, there exists a weak solution $u \in L^2(\mathbb{\C}, \mathrm{\textnormal{e}}^{- \vert z \vert^2})$ of the equation
$$ \bar{\partial}^{k} u + c' u =  f$$ where $c'= \frac{c}{\beta}$ with the estimation of the norm
$$\int_\mathbb{C} \vert u \vert^2 \mathrm{\textnormal{e}}^{- \vert z \vert^2} d\sigma \leq \frac{1}{(k!)} \int_\mathbb{C} \vert f \vert^2 \mathrm{\textnormal{e}}^{- \vert z \vert^2} d\sigma$$ 
so $$\beta  \bar{\partial}^{k} u + cu = \beta f.$$
Or $\vert \beta \vert \geq 1$ so
$$\frac{1}{(k!)} \int_\mathbb{C} \vert f \vert^2 \mathrm{\textnormal{e}}^{- \vert z \vert^2} d\sigma \leq \frac{\vert \beta \vert^2}{(k!)} \int_\mathbb{C} \vert f \vert^2 \mathrm{\textnormal{e}}^{- \vert z \vert^2} d\sigma$$
Thus
$$\int_\mathbb{C} \vert u \vert^2 \mathrm{\textnormal{e}}^{- \vert z \vert^2} d\sigma \leq \frac{\vert \beta \vert^2}{(k!)} \int_\mathbb{C} \vert f \vert^2 \mathrm{\textnormal{e}}^{- \vert z \vert^2} d\sigma$$
\end{proof}
\begin{theorem}
There is a linear and bounded operator $$T_k : L^2(\mathbb{\C}, \mathrm{\textnormal{e}}^{- \vert z \vert^2}) \longrightarrow L^2(\mathbb{\C}, \mathrm{\textnormal{e}}^{- \vert z \vert^2})$$
such that $$ (\beta \bar{\partial}^{k}  + c) T_k = I$$ with $\vert \alpha \vert \geq 1$ and the estimation of the norm
$$\vert \vert T_k \vert \vert^2_\varphi \leq \frac{1}{(k!)}$$ where $\vert \vert T_k \vert \vert_\varphi$ is the norm of $T_k$ in $L^2(\mathbb{\C}, \mathrm{\textnormal{e}}^{- \vert z \vert^2})$.
\end{theorem}
\begin{proof}
Let $f \in L^2(\mathbb{\C}, \mathrm{\textnormal{e}}^{- \vert z \vert^2})$. According to Theorem \ref{q}, there exists a weak solution $u \in L^2(\mathbb{\C}, \mathrm{\textnormal{e}}^{- \vert z \vert^2})$ such that
$$\beta  \bar{\partial}^{k} u + cu = \beta f$$  with the estimation of the norm 
 $$\vert \vert u \vert \vert^2_\varphi \leq \frac{\vert \beta \vert^2}{(k!)} \vert \vert f \vert \vert^2_\varphi$$
 Or $\beta f \in L^2(\mathbb{\C}, \mathrm{\textnormal{e}}^{- \vert z \vert^2})$\\
so
$$ (\beta \bar{\partial}^{k}  + c)T_k(\beta f) = \beta f$$ with
 $$\vert \vert T_k(\beta f) \vert \vert^2_\varphi \leq \frac{\vert \beta \vert^2}{(k!)} \vert \vert f \vert \vert^2_\varphi.$$
 Thus
 $$ (\beta  \bar{\partial}^{k}  + c)T_k = I$$ with
 $$\vert \vert T_k \vert \vert^2_\varphi \leq \frac{1}{(k!)}.$$
\end{proof}
\begin{thm}\label{t1}
Let $\varphi$ be a smooth and strictly positive function on $\mathbb{C}$ with $\Delta \varphi > 0$. For all $ f \in L^2(\mathbb{\C}, \mathrm{\textnormal{e}}^{-\varphi})$ such that $$\frac{f}{\sqrt{\Delta \varphi }}  \in L^2(\mathbb{\C}, \mathrm{\textnormal{e}}^{-\varphi}),$$ there exists a weak solution $ u \in  L^2(\mathbb{\C}, \mathrm{\textnormal{e}}^{-\varphi})$ of the equation
$$\beta  \bar{\partial}^1 u + cu = \beta f$$ with $\vert \beta \vert \geq 1$ and the estimation of the norm
$$\int_\mathbb{C} \vert u \vert^2 \mathrm{\textnormal{e}}^{-\varphi} d\sigma \leq 4  \vert \beta \vert^2 \int_\mathbb{C} \frac{\vert f \vert^2}{\Delta \varphi}  \mathrm{\textnormal{e}}^{-\varphi} d\sigma.$$
\end{thm}
\begin{proof}
Let $ f \in L^2(\mathbb{\C}, \mathrm{\textnormal{e}}^{-\varphi})$ such that  $$\frac{f}{\sqrt{\Delta \varphi }}  \in L^2(\mathbb{\C}, \mathrm{\textnormal{e}}^{-\varphi}),$$
according to Theorem $1.2$ of \cite{1}, there exists a weak solution $u \in L^2(\mathbb{\C}, \mathrm{\textnormal{e}}^{-\varphi})$ of the equation
$$  \bar{\partial}^1 u + c' u =  f$$ o\`u $c'= \frac{c}{\beta}$ with the estimation of the norm 
$$\int_\mathbb{C} \vert u \vert^2 \mathrm{\textnormal{e}}^{-\varphi} d\sigma \leq 4  \int_\mathbb{C} \frac{\vert f \vert^2}{\Delta \varphi}  \mathrm{\textnormal{e}}^{-\varphi} d\sigma.$$
so $$\beta  \bar{\partial}^1 u + cu = \beta f.$$
Or $\vert \beta \vert \geq 1$ so
$$4  \int_\mathbb{C} \frac{\vert f \vert^2}{\Delta \varphi}  \mathrm{\textnormal{e}}^{-\varphi} d\sigma \leq  4  \vert \beta \vert^2 \int_\mathbb{C} \frac{\vert f \vert^2}{\Delta \varphi}  \mathrm{\textnormal{e}}^{-\varphi} d\sigma.$$
Thus
$$\int_\mathbb{C} \vert u \vert^2 \mathrm{\textnormal{e}}^{-\varphi} d\sigma \leq 4 \vert \beta \vert^2 \int_\mathbb{C} \frac{\vert f \vert^2}{\Delta \varphi }  \mathrm{\textnormal{e}}^{-\varphi} d\sigma.$$
\end{proof}
%%%%%%%%%%%%%%%%

\subsection{Some consequences of the case of the operator $\beta  \bar{\partial}^{k} + c$}
%\; \\
Let $\lambda > 0$ and $z_0 \in \mathbb{C}$. Let $\varphi = \lambda \vert z - z_0 \vert^2$, we get the following corollary of Theorem \ref{q} :
\begin{coro} \label{lam2}
Let $f \in L^2(\mathbb{\C}, \mathrm{\textnormal{e}}^{-\lambda \vert z - z_0 \vert^2})$, there exists a weak solution $u \in L^2(\mathbb{\C}, \mathrm{\textnormal{e}}^{-\lambda \vert z - z_0 \vert^2})$ of the equation
$$\beta \bar{\partial}^{k} u + cu = \beta f$$ with $\vert \beta \vert \geq 1$ and the estimation of the norm
$$\int_\mathbb{C} \vert u \vert^2 \mathrm{\textnormal{e}}^{- \lambda \vert z - z_0 \vert^2} d\sigma \leq \frac{\vert \beta \vert^2}{\lambda^k k!} \int_\mathbb{C} \vert f \vert^2 \mathrm{\textnormal{e}}^{- \lambda \vert z - z_0 \vert^2} d\sigma$$
\end{coro}
\begin{proof}
Let $f \in L^2(\mathbb{\C}, \mathrm{\textnormal{e}}^{-\lambda \vert z - z_0 \vert^2})$, we have $$\int_\mathbb{C} \vert f \vert^2 \mathrm{\textnormal{e}}^{- \lambda \vert z - z_0 \vert^2} d\sigma < + \infty.$$
Let $z = \frac{w}{\lambda} + z_0$ and $g(w)= f(z)= f(\frac{w}{\lambda} + z_0)$ then
$$\frac{1}{\lambda} \int_\mathbb{C} \vert g \vert^2 \mathrm{\textnormal{e}}^{- \vert w \vert^2} d\sigma(w) < + \infty.$$
so $g \in L^2(\mathbb{\C}, \mathrm{\textnormal{e}}^{-\vert w \vert^2})$. According to Theorem \ref{q}, there exists a weak solution $v \in L^2(\mathbb{\C}, \mathrm{\textnormal{e}}^{-\vert w \vert^2})$ of the equation
$$\beta \bar{\partial}^{k} v +  \frac{c}{(\sqrt{\lambda})^k}v = \beta g$$ with the estimation of the norm 
$$\int_\mathbb{C} \vert v \vert^2 \mathrm{\textnormal{e}}^{- \vert w \vert^2} d\sigma(w) \leq \frac{\vert \beta \vert^2}{k!} \int_\mathbb{C} \vert g \vert^2 \mathrm{\textnormal{e}}^{- \vert w \vert^2} d\sigma(w).$$
Let $u(z) = \frac{1}{(\sqrt{\lambda})^k} v(w) = \frac{1}{(\sqrt{\lambda})^k} v(\sqrt{\lambda} (z - z_0))$.\\
Then 
$$\beta \bar{\partial}^{k} u + cu = \beta f$$ with the estimation of the norm
$$\int_\mathbb{C} \vert u \vert^2 \mathrm{\textnormal{e}}^{- \lambda \vert z - z_0 \vert^2} d\sigma \leq \frac{\vert \beta \vert^2}{\lambda^k k!} \int_\mathbb{C} \vert f \vert^2 \mathrm{\textnormal{e}}^{- \lambda \vert z - z_0 \vert^2} d\sigma.$$
\end{proof}\\
%%%%%%%%%%%%%%

As a consequence of Corollary \ref{lam2} we have :
\begin{coro} \label{4}
Let $U \in \mathbb{C}$ be a bounded open set.\\
Let $f \in L^2(U)$, there exists a weak solution $u \in L^2(U)$ of the equation
$$\beta  \bar{\partial}^{k} u + cu = \beta f$$ with $\vert \beta \vert \geq 1$ and
$$\vert \vert u \vert \vert_{L^2(U)} \leq \Big(\frac{{\textnormal{e}}^{\vert U \vert^2}\vert \beta \vert^2}{k!}\Big) \vert \vert f \vert \vert_{L^2(U)}$$ 
where $\vert U \vert $ is the diameter of $U$.
\end{coro}
\begin{proof}
Let $z_0 \in U$ and $f \in L^2(U)$, we have
 \[ 
 \tilde{f}(z) = \left \{
 \begin{array}{rl}
 f(z) & \mbox{ if } z \in U \\
 0 & \mbox{ if } z \in \mathbb{C} \setminus U
 \end{array}
 \right.
 \]
 So $\tilde{f} \in L^2(\mathbb{C}) \subset L^2(\mathbb{\C}, \mathrm{\textnormal{e}}^{-\vert z - z_0 \vert^2})$. According to Corolllary \ref{lam2}, there exists a weak solution $\tilde{u} \in  L^2(\mathbb{\C}, \mathrm{\textnormal{e}}^{-\vert z - z_0 \vert^2})$ of the equation $$\beta \bar{\partial}^{k} \tilde{u} + c\tilde{u} = \beta \tilde{f}$$ with
 $$\int_\mathbb{C} \vert \tilde{u} \vert^2 \mathrm{\textnormal{e}}^{-  \vert z - z_0 \vert^2} d\sigma \leq \frac{\vert \beta \vert^2}{k!} \int_\mathbb{C} \vert \tilde{f} \vert^2 \mathrm{\textnormal{e}}^{-  \vert z - z_0 \vert^2} d\sigma$$
 so
 $$\int_\mathbb{C} \vert \tilde{u} \vert^2 \mathrm{\textnormal{e}}^{-  \vert z - z_0 \vert^2} d\sigma \leq \frac{\vert \beta \vert^2}{k!} \int_U \vert f \vert^2  d\sigma$$
 or $$\int_\mathbb{C} \vert \tilde{u} \vert^2 \mathrm{\textnormal{e}}^{-  \vert z - z_0 \vert^2} d\sigma \geq \mathrm{\textnormal{e}}^{-\vert U \vert^2} \int_U \vert \tilde{u} \vert^2  d\sigma$$
 thus
 $$\int_U \vert \tilde{u} \vert^2  d\sigma \leq \Big(\frac{{\textnormal{e}}^{\vert U \vert^2}\vert \beta \vert^2}{k!}\Big) \int_U \vert f \vert^2  d\sigma$$
 or $\tilde{u}_{\mid U} = u$ then we have
 $$\beta  \bar{\partial}^{k} u + cu = \beta f$$ with
$$\vert \vert u \vert \vert_{L^2(U)} \leq \Big(\frac{{\textnormal{e}}^{\vert U \vert^2}\vert \beta \vert^2}{k!}\Big)\vert \vert f \vert \vert_{L^2(U)}.$$ 
\end{proof}
\begin{remark}
Case where $\alpha= \beta=0$:  The operator $\gamma \partial^{k}  + c$ is the conjugate of the case where $\alpha= \gamma=0$:  The operator $\beta \bar{\partial}^{k}  + c$. 
\end{remark}
%%%%%%%%%%%%

%\begin{coro}
%Soit $f \in L^p(\mathbb{C})$ avec $p \geq 2$, il existe une solution faible $u \in L_{loc}^2(\mathbb{C})$ de l'équation
%$$\alpha \bar{\partial}^{k} u + cu = \alpha f$$ avec $\vert \alpha \vert \geq 1$.
%\end{coro}
%\begin{proof}
%Soit $f \in L^p(\mathbb{C})$ avec $p \geq 2$.\\
%Or $L^p(\mathbb{C}) \subset  L^2(\mathbb{\C}, \mathrm{\textnormal{e}}^{-\vert z  \vert^2})$,
%\; \; \; \;  $L^p(\mathbb{C}) \subset  L^2(\mathbb{\C}, \mathrm{\textnormal{e}}^{-\vert z  \vert^2})$\\ et $$L^2(\mathbb{\C}, \mathrm{\textnormal{e}}^{-\vert z  \vert^2}) \subset L_{loc}^2(\mathbb{C})$$
%donc d'après le Théorème \ref{q}, il existe une solution faible $u \in L_{loc}^2(\mathbb{C})$ de l'équation
%$$\alpha \bar{\partial}^{k} u + cu = \alpha f$$
%\end{proof}
%%%%%%%%%%%%%%%%%%%%%%%%%%%%%%%%%%%%%%%%%%%%
\subsection{Case where $\beta= \gamma=0$: The operator  $\alpha \partial^{k} \bar{\partial}^{k}  + c$}
%\;\\
Taking into account Theorem  $1.1$ of \cite{3}, we have the following result :
\begin{thm} \label{p}
For all $f \in L^2(\mathbb{\C}, \mathrm{\textnormal{e}}^{- \vert z \vert^2})$, there exists a weak solution $u \in L^2(\mathbb{\C}, \mathrm{\textnormal{e}}^{- \vert z \vert^2})$ of the equation
$$\alpha \partial^{k} \bar{\partial}^{k} u + cu = \alpha f$$ with $\vert \alpha \vert \geq 1$ and the estimation of the norm
$$\int_\mathbb{C} \vert u \vert^2 \mathrm{\textnormal{e}}^{- \vert z \vert^2} d\sigma \leq \frac{\vert \alpha \vert^2}{(k!)^2} \int_\mathbb{C} \vert f \vert^2 \mathrm{\textnormal{e}}^{- \vert z \vert^2} d\sigma$$
\end{thm}
\begin{proof}
Let $f \in L^2(\mathbb{\C}, \mathrm{\textnormal{e}}^{- \vert z \vert^2})$, according to Theorem $1.1$ of \cite{3}, there exists a weak solution $u \in L^2(\mathbb{\C}, \mathrm{\textnormal{e}}^{- \vert z \vert^2})$ of the equation
$$ \partial^k \bar{\partial}^{k} u + c' u =  f$$ where $c'= \frac{c}{\alpha}$ with the estimation of the norm
$$\int_\mathbb{C} \vert u \vert^2 \mathrm{\textnormal{e}}^{- \vert z \vert^2} d\sigma \leq \frac{1}{(k!)^2} \int_\mathbb{C} \vert f \vert^2 \mathrm{\textnormal{e}}^{- \vert z \vert^2} d\sigma$$ 
so $$\alpha \partial^k \bar{\partial}^{k} u + cu = \alpha f.$$
Or $\vert \alpha \vert \geq 1$ so
$$\frac{1}{(k!)^2} \int_\mathbb{C} \vert f \vert^2 \mathrm{\textnormal{e}}^{- \vert z \vert^2} d\sigma \leq \frac{\vert \alpha \vert^2}{(k!)^2} \int_\mathbb{C} \vert f \vert^2 \mathrm{\textnormal{e}}^{- \vert z \vert^2} d\sigma$$
Thus
$$\int_\mathbb{C} \vert u \vert^2 \mathrm{\textnormal{e}}^{- \vert z \vert^2} d\sigma \leq \frac{\vert \alpha \vert^2}{(k!)^2} \int_\mathbb{C} \vert f \vert^2 \mathrm{\textnormal{e}}^{- \vert z \vert^2} d\sigma$$
\end{proof}
As a consequence of Theorem \ref{p} we have :
\begin{theorem}
There is a linear and bounded operator $$T_k : L^2(\mathbb{\C}, \mathrm{\textnormal{e}}^{- \vert z \vert^2}) \longrightarrow L^2(\mathbb{\C}, \mathrm{\textnormal{e}}^{- \vert z \vert^2})$$
such that $$ (\alpha \partial^k \bar{\partial}^{k}  + c) T_k = I$$ with $\vert \alpha \vert \geq 1$ and the estimation of the norm
$$\vert \vert T_k \vert \vert^2_\varphi \leq \frac{1}{(k!)^2}$$ where $\vert \vert T_k \vert \vert_\varphi$ is the norm of $T_k$ in $L^2(\mathbb{\C}, \mathrm{\textnormal{e}}^{- \vert z \vert^2})$.
\end{theorem}
\begin{proof}
Let $f \in L^2(\mathbb{\C}, \mathrm{\textnormal{e}}^{- \vert z \vert^2})$. According to Theorem \ref{p}, there exists a weak solution $u \in L^2(\mathbb{\C}, \mathrm{\textnormal{e}}^{- \vert z \vert^2})$ such that
$$\alpha \partial^k \bar{\partial}^{k} u + cu = \alpha f$$ with the estimation of the norm
 $$\vert \vert u \vert \vert^2_\varphi \leq \frac{\vert \alpha \vert^2}{(k!)^2} \vert \vert f \vert \vert^2_\varphi$$
 Or $\alpha f \in L^2(\mathbb{\C}, \mathrm{\textnormal{e}}^{- \vert z \vert^2})$\\
so 
$$ (\alpha \partial^k \bar{\partial}^{k}  + c)T_k(\alpha f) = \alpha f$$ avec 
 $$\vert \vert T_k(\alpha f) \vert \vert^2_\varphi \leq \frac{\vert \alpha \vert^2}{(k!)^2} \vert \vert f \vert \vert^2_\varphi.$$
 Thus
 $$ (\alpha \partial^k \bar{\partial}^{k}  + c)T_k = I$$ with 
 $$\vert \vert T_k \vert \vert^2_\varphi \leq \frac{1}{(k!)^2}.$$
\end{proof}
\vskip 0.5mm
We have also the case $k = 1$ of Theorem \ref{p} stated by the following result.
\begin{thm}\label{t1}
Let $\varphi$ be a smooth and strictly positive function on $\mathbb{C}$ with $\Delta(\mathrm{\textnormal{e}}^{\varphi} \Delta \mathrm{\textnormal{e}}^{-\varphi})> 0$. For all $ f \in L^2(\mathbb{\C}, \mathrm{\textnormal{e}}^{-\varphi})$ such that $$\frac{f}{\sqrt{\Delta(\mathrm{\textnormal{e}}^{\varphi} \Delta \mathrm{\textnormal{e}}^{-\varphi })}}  \in L^2(\mathbb{\C}, \mathrm{\textnormal{e}}^{-\varphi }),$$ there exists a weak solution $ u \in  L^2(\mathbb{\C}, \mathrm{\textnormal{e}}^{-\varphi})$ of the equation
$$\alpha \partial^1 \bar{\partial}^1 u + cu = \alpha f$$ with $\vert \alpha \vert \geq 1$ and the estimation of the norm
$$\int_\mathbb{C} \vert u \vert^2 \mathrm{\textnormal{e}}^{-\varphi} d\sigma \leq 16  \vert \alpha \vert^2 \int_\mathbb{C} \frac{\vert f \vert^2}{\Delta(\mathrm{\textnormal{e}}^{\varphi} \Delta \mathrm{\textnormal{e}}^{-\varphi})}  \mathrm{\textnormal{e}}^{-\varphi} d\sigma.$$
\end{thm}
\begin{proof}
Let $ f \in L^2(\mathbb{\C}, \mathrm{\textnormal{e}}^{-\varphi})$ such that $$\frac{f}{\sqrt{\Delta(\mathrm{\textnormal{e}}^{\varphi} \Delta \mathrm{\textnormal{e}}^{-\varphi})}}  \in L^2(\mathbb{\C}, \mathrm{\textnormal{e}}^{-\varphi}),$$
according to the Theorem $1.2$ of \cite{3}, there exists a weak solution $u \in L^2(\mathbb{\C}, \mathrm{\textnormal{e}}^{-\varphi})$ of the equation
$$ \partial^1 \bar{\partial}^1 u + c' u =  f$$ where $c'= \frac{c}{\alpha}$ with the estimation of the norm 
$$\int_\mathbb{C} \vert u \vert^2 \mathrm{\textnormal{e}}^{-\varphi } d\sigma \leq 16  \int_\mathbb{C} \frac{\vert f \vert^2}{\Delta(\mathrm{\textnormal{e}}^{\varphi} \Delta \mathrm{\textnormal{e}}^{-\varphi})}  \mathrm{\textnormal{e}}^{-\varphi} d\sigma.$$
so $$\alpha \partial^1 \bar{\partial}^1 u + cu = \alpha f.$$
Or $\vert \alpha \vert \geq 1$ so
$$16  \int_\mathbb{C} \frac{\vert f \vert^2}{\Delta(\mathrm{\textnormal{e}}^{\varphi} \Delta \mathrm{\textnormal{e}}^{-\varphi})}  \mathrm{\textnormal{e}}^{-\varphi} d\sigma \leq  16  \vert \alpha \vert^2 \int_\mathbb{C} \frac{\vert f \vert^2}{\Delta(\mathrm{\textnormal{e}}^{\varphi} \Delta \mathrm{\textnormal{e}}^{-\varphi})}  \mathrm{\textnormal{e}}^{-\varphi} d\sigma.$$
Thus
$$\int_\mathbb{C} \vert u \vert^2 \mathrm{\textnormal{e}}^{-\varphi} d\sigma \leq 16  \vert \alpha \vert^2 \int_\mathbb{C} \frac{\vert f \vert^2}{\Delta(\mathrm{\textnormal{e}}^{\varphi} \Delta \mathrm{\textnormal{e}}^{-\varphi})}  \mathrm{\textnormal{e}}^{-\varphi} d\sigma.$$
\end{proof}
\subsection{Some consequences of the case of the operator $\alpha \partial^{k}  \bar{\partial}^{k} + c$}
%\; \\
Let $\lambda > 0$ and $z_0 \in \mathbb{C}$. Let $\varphi = \lambda^2 \vert z - z_0 \vert^2$, we get the following corollary of Theorem \ref{p} :
\begin{coro} \label{lam}
Let $f \in L^2(\mathbb{\C}, \mathrm{\textnormal{e}}^{-\lambda^2 \vert z - z_0 \vert^2})$, there exists a weak solution $u \in L^2(\mathbb{\C}, \mathrm{\textnormal{e}}^{-\lambda^2 \vert z - z_0 \vert^2})$ with the equation
$$\alpha \partial^k \bar{\partial}^{k} u + cu = \alpha f$$ with $\vert \alpha \vert \geq 1$ and the estimation of the norm
$$\int_\mathbb{C} \vert u \vert^2 \mathrm{\textnormal{e}}^{- \lambda^2 \vert z - z_0 \vert^2} d\sigma \leq \frac{\vert \alpha \vert^2}{(\lambda^k k!)^2} \int_\mathbb{C} \vert f \vert^2 \mathrm{\textnormal{e}}^{- \lambda^2 \vert z - z_0 \vert^2} d\sigma.$$
\end{coro}
\begin{proof}
Let $f \in L^2(\mathbb{\C}, \mathrm{\textnormal{e}}^{-\lambda^2 \vert z - z_0 \vert^2})$, we have $$\int_\mathbb{C} \vert f \vert^2 \mathrm{\textnormal{e}}^{- \lambda^2 \vert z - z_0 \vert^2} d\sigma < + \infty.$$
Let $z = \frac{w}{\lambda} + z_0$ and $g(w)= f(z)= f(\frac{w}{\lambda} + z_0)$ then
$$\frac{1}{\lambda^2} \int_\mathbb{C} \vert g \vert^2 \mathrm{\textnormal{e}}^{- \vert w \vert^2} d\sigma(w) < + \infty.$$
so $g \in L^2(\mathbb{\C}, \mathrm{\textnormal{e}}^{-\vert w \vert^2})$. According to Theorem \ref{p}, there exists a weak solution $v \in L^2(\mathbb{\C}, \mathrm{\textnormal{e}}^{-\vert w \vert^2})$ of the equation
$$\alpha \partial^k \bar{\partial}^{k} v + \frac{c}{\lambda^k}v = \alpha g$$ with the estimation of norm
$$\int_\mathbb{C} \vert v \vert^2 \mathrm{\textnormal{e}}^{- \vert w \vert^2} d\sigma(w) \leq \frac{\vert \alpha \vert^2}{( k!)^2} \int_\mathbb{C} \vert g \vert^2 \mathrm{\textnormal{e}}^{- \vert w \vert^2} d\sigma(w).$$
Let $u(z) = \frac{1}{\lambda^k} v(w) = \frac{1}{\lambda^k} v(\lambda (z - z_0))$.\\
Then 
$$\alpha \partial^k \bar{\partial}^{k} u + cu = \alpha f$$ with the estimation of the norm
$$\int_\mathbb{C} \vert u \vert^2 \mathrm{\textnormal{e}}^{- \lambda^2 \vert z - z_0 \vert^2} d\sigma \leq \frac{\vert \alpha \vert^2}{(\lambda^k k!)^2} \int_\mathbb{C} \vert f \vert^2 \mathrm{\textnormal{e}}^{- \lambda^2 \vert z - z_0 \vert^2} d\sigma.$$
\end{proof}

As a consequence of Corollary \ref{lam} we have :
\begin{coro} \label{3}
Let $U \subset \mathbb{C}$ be a bonded open set.\\
Let $f \in L^2(U)$, there exists a weak solution $u \in L^2(U)$ of the equation
$$\alpha \partial^k \bar{\partial}^{k} u + cu = \alpha f$$ with $\vert \alpha \vert \geq 1$ and
$$\vert \vert u \vert \vert_{L^2(U)} \leq \Big(\frac{{\textnormal{e}}^{\vert U \vert^2}\vert \alpha \vert^2}{(k!)^2}\Big) \vert \vert f \vert \vert_{L^2(U)}$$ 
where $\vert U \vert $ is the diameter of $U$.
\end{coro}
\begin{proof}
Let $z_0 \in U$ and $f \in L^2(U)$, we have
 \[ 
 \tilde{f}(z) = \left \{
 \begin{array}{rl}
 f(z) & \mbox{ if } z \in U \\
 0 & \mbox{ if } z \in \mathbb{C} \setminus U
 \end{array}
 \right.
 \]
 So $\tilde{f} \in L^2(\mathbb{C}) \subset L^2(\mathbb{\C}, \mathrm{\textnormal{e}}^{-\vert z - z_0 \vert^2})$. According to Corollary \ref{lam}, there exists a weak solution $\tilde{u} \in L^2(\mathbb{\C}, \mathrm{\textnormal{e}}^{-\vert z - z_0 \vert^2})$ of the equation $$\alpha \partial^k \bar{\partial}^{k} \tilde{u} + c\tilde{u} = \alpha \tilde{f}$$ with
 $$\int_\mathbb{C} \vert \tilde{u} \vert^2 \mathrm{\textnormal{e}}^{-  \vert z - z_0 \vert^2} d\sigma \leq \frac{\vert \alpha \vert^2}{(k!)^2} \int_\mathbb{C} \vert \tilde{f} \vert^2 \mathrm{\textnormal{e}}^{-  \vert z - z_0 \vert^2} d\sigma$$
 so
 $$\int_\mathbb{C} \vert \tilde{u} \vert^2 \mathrm{\textnormal{e}}^{-  \vert z - z_0 \vert^2} d\sigma \leq \frac{\vert \alpha \vert^2}{(k!)^2} \int_U \vert f \vert^2  d\sigma$$
 or $$\int_\mathbb{C} \vert \tilde{u} \vert^2 \mathrm{\textnormal{e}}^{-  \vert z - z_0 \vert^2} d\sigma \geq \mathrm{\textnormal{e}}^{-\vert U \vert^2} \int_U \vert \tilde{u} \vert^2  d\sigma$$
 then
 $$\int_U \vert \tilde{u} \vert^2  d\sigma \leq \Big(\frac{{\textnormal{e}}^{\vert U \vert^2}\vert \alpha \vert^2}{(k!)^2}\Big) \int_U \vert f \vert^2  d\sigma$$
 or $\tilde{u}_{\mid U} = u$ then we have
 $$\alpha \partial^k \bar{\partial}^{k} u + cu = \alpha f$$ with
$$\vert \vert u \vert \vert_{L^2(U)} \leq \Big(\frac{{\textnormal{e}}^{\vert U \vert^2}\vert \alpha \vert^2}{(k!)^2}\Big)\vert \vert f \vert \vert_{L^2(U)}.$$ 
\end{proof}

\end{document}